\documentclass[a4paper, reqno]{amsart}
\usepackage{color}

\newtheorem{theorem}{Theorem}[section]
\newtheorem{lemma}[theorem]{Lemma}

\newtheorem{notations}[theorem]{Notations}

\theoremstyle{remark}

\newtheorem{remark}[theorem]{Remark}
\newtheorem{definition}[theorem]{Definition}

\numberwithin{equation}{section}

\def\N{{\mathbb{N}}}

\def\R{{\mathbb{R}}}
\def\C{{\mathbb{C}}}
\def\D{{\mathbb{D}}}
\def\T{{\mathbb{T}}}
\def\sH{{\mathcal{H}}}
\def\sK{{\mathcal{K}}}
\def\mX{{\mathcal{X}}}
\def\tr{{\mathrm{Tr}\,}}

\def\mU{{\mathcal{U}}}
\renewcommand{\i}{\text{\rm i}}

\begin{document}

\title{Higher order spectral shift for contractions}

\author[Potapov]{Denis Potapov$^{*}$}
\address{School of Mathematics and Statistics, University of New South Wales, Kensington, NSW 2052, Australia}
\email{d.potapov@unsw.edu.au}

\author[Skripka]{Anna Skripka$^{**}$}
\address{Department of Mathematics and Statistics, University of New Mexico,
400 Yale Blvd NE, MSC01 1115,
Albuquerque, NM 87131-0001}
\email{skripka@math.uum.edu}

\author[Sukochev]{Fedor Sukochev$^{*}$} \address{School of Mathematics
  and Statistics, University of New South Wales, Kensington, NSW 2052,
  Australia} \email{f.sukochev@unsw.edu.au}

\thanks{\footnotesize $^{*}$Research supported in part by ARC}
\thanks{\footnotesize $^{**}$Research supported in part by NSF grant
  DMS-0900870}



\begin{abstract}
We derive strong estimates for Schatten norms of operator derivatives along paths of contractions and apply them to prove existence of higher order spectral shift functions for pairs of contractions.
\end{abstract}

\maketitle

\section{Introduction.}
Let $A$ and $B$ be bounded operators on a separable Hilbert space $\mathcal{H}$ and $f$ be a sufficiently smooth function for which the functional calculus $f(A)$, $f(B)$ makes sense. The question concerning conditions on the operators $A$ and $B$ and the function $f$ under which the trace formula
\begin{align}
\label{tf0}
\tr\big(f(A)-f(B)\big)=\int f'(t)\xi(t)\,dt
\end{align} holds (here $\xi$ depends only on $A$, $B$ and the integration is taken over a suitable domain) can be traced to M.~G.~Krein's penetrating papers \cite{Krein,KreinUn,KY,KreinOT}.

In case when the difference $f(A)-f(B)$ is not in the trace class, S.~L.~Koplienko suggested in \cite{Koplienko} to modify the formula \eqref{tf0} as follows:
\begin{equation}
    \label{nTraceFormula}
    \tr \left( f(A)-\sum_{k = 0}^{n-1} \frac 1{k!}\, \frac {d^k}{dt^k}
      \left[ f(B + t(A-B)) \right]\Bigr|_{t = 0} \right) = \int_\R f^{(n)}
    (t)\, \eta_n  (t)\, dt,
  \end{equation} where $\eta_n$ depends only on $A,B$, and $n\in\N$. The question of validity of the formula \eqref{nTraceFormula} was later investigated for various classes of operators $A$ and $B$ in \cite{GPS,Neidhardt,PellerKo,PS-circle,PSS}.

In this paper, we answer the latter question by proving that \eqref{nTraceFormula} holds for $A$ and $B$ arbitrary contractions (with the minimal restriction on $A-B$ to guarantee that the left hand side of \eqref{nTraceFormula} is well defined) and $f$ being a polynomial.\footnote{The formula \eqref{nTraceFormula} can be extended to more general functions $f$; however, we do not address this extension for the sake of less involved exposition.} To realize this goal, we establish powerful estimates for derivatives of operator functions along paths of contractions which are of independent interest. Our proof involves subtle synthesis of ideas from recent advances on multiple operator integration for self-adjoint operators \cite{PSS}, double operator integration of functions of several variables \cite{KPSS}, and application of classical theory of analytic functions in perturbation theory in the spirit of \cite{PS-circle} as well as develops new interesting methods. 

We proceed with a more detailed description of the history of the question and our main results.

It is known that given a pair of self-adjoint operators $H_0$ and $V$ on a separable Hilbert space, with $V$ in the Schatten-von Neumann ideal $S^n$, $n\in\N$, there exists a function $\eta_n$, called $n$th order spectral shift function (SSF), depending on $n,H_0,V$ such that \eqref{nTraceFormula} holds with $B=H_0$ and $A=H_0+V$.
The cases $n=1$, $n=2$, and $n\geq 3$ are due to M.~G.~Krein \cite{Krein}, L.~S.~Koplienko \cite{Koplienko}, and the authors \cite{PSS}, respectively. The formula \eqref{nTraceFormula} has been extended from the original set of functions $f$ to the Besov class $B_{\infty1}^n$ in \cite{PellerKr,PellerKo,AP}.

Existence of the first and second order spectral shift functions for a pair of unitaries $U_0$ and $U_0+V$ was established in \cite{KreinUn} in case $V\in S^1$ and in \cite{Neidhardt} in case $V\in S^2$, respectively, but it has taken longer than in the self-adjoint case to find plausible SSFs for pairs of contractions. References on partial results for specific pairs of contractions can be found in \cite{AN,PS-circle}. Existence of the second order integrable SSF for any pair of contractions $U_0$ and $U_0+V$ with $V\in S^2$ has recently been proved in \cite{PS-circle}. The latter paper answers \cite[Question 11.2]{GPS} for $V\in S^2$, and in this paper we obtain the result for much more general perturbations $V\in S^n$, $n\geq 3$. More precisely, we prove existence of an integrable higher order spectral shift function for {\it any} pair of contractions $U_0$ and $U_0+V$ with the perturbation $V\in S^n$, $n\geq 3$.


We fix our main notations below.
\begin{notations}
\label{hyp}
\begin{enumerate}
\item
Let $U_0,U_1$ be contractions on a separable Hilbert space $\sH$. Denote $V:=U_1-U_0$ and define the path of contractions \label{hyp1} \[U_t:=(1-t)U_0+tU_1=U_0+tV,\quad t\in[0,1].\] 
\item
Let $n\in\N$. For $f$ a polynomial, denote \label{hyp2}
\[R_n(f,U_0,V):=f(U_0+V)-\sum_{j=0}^{n-1}\frac{1}{j!}\frac{d^j}{dt^j}f(U_0+tV)\bigg|_{t=0}.\]
\end{enumerate}
\end{notations}

Our main results are the following two theorems.

\begin{theorem}
\label{thm:MainEst}
Assume Notations \ref{hyp} \eqref{hyp1}. Then, there exists a constant $c_n>0$ such that for any polynomial $f$ the following estimates hold.
\begin{enumerate}

\item If $\alpha>n$ and $V\in S^\alpha$, then
\begin{equation}
\label{eq:MainEst}
\sup_{t_0\in [0,1]}\left\|\frac{d^n}{dt^n}f(U_0+tV)\big|_{t=t_0}\right\|_{\frac{\alpha}{n}}\leq c_n \|f^{(n)}\|_{L^\infty(\T)}\|V\|_\alpha^{n}.
\end{equation}

\item If $V\in S^n$, then
\begin{equation}
\label{eq:TrEst}
\sup_{t_0\in [0,1]}\left|\tr\left(\frac{d^n}{dt^n}f(U_0+tV)\big|_{t=t_0}\right)\right|\leq c_n \|f^{(n)}\|_{L^\infty(\T)}\|V\|_n^n.
\end{equation}

\end{enumerate}
\end{theorem}

Differentiation of analytic Besov functions of contractions was discussed in \cite{PellerContr}, but estimates for operator derivatives that follow from the results in \cite{PellerContr}
\[\sup_{t_0\in [0,1]}\left|\tr\left(\frac{d^n}{dt^n}f(U_0+tV)\big|_{t=t_0}\right)\right|\leq c_n \big\|f^{[n]}\big\|_{\otimes}\|V\|_n^n.\]
contain a factor of the integral projective tensor product norm $\big\|f^{[n]}\big\|_{\otimes}$ of the $n$th order divided difference $f^{[n]}$ of $f$, which is greater than the norm $\|f^{(n)}\|_\infty$, while the estimates with $\|f^{(n)}\|_\infty$ are needed in the proof of existence of higher order spectral shift functions.

As a consequence of Theorem \ref{thm:MainEst}, we establish existence of the higher order spectral shift functions for pairs of contractions.

\begin{theorem}
\label{existence}
Assume Notations \ref{hyp} and assume $V\in S^n$. Then, there exists a function $\eta_n=\eta_{n,U_0,V}$ in $L^1(\T)$ such that
\begin{align}
\label{hossf}
\tr\big(R_n(f,U_0,V)\big)=\int_\T f^{(n)}(z)\eta_n(z)\,dz.
\end{align}
Furthermore, for every given $\epsilon>0$, the function $\eta_n$ satisfying \eqref{hossf} can be chosen so that
\begin{align}
\label{ssfest}
\|\eta_n\|_1\leq (1+\epsilon)c_n\|V\|_n^n,
\end{align}
where $c_n$ is a constant from Theorem \ref{thm:MainEst}.
\end{theorem}

The proof of Theorem \ref{thm:MainEst} is given in Section \ref{sec2} and the proof of Theorem \ref{existence} in Section \ref{sec3}. Theorem \ref{thm:MainEst} in case of contractions naturally reduces to the case of unitaries, while the case of unitaries requires a very sophisticated treatment. Although our main results are
analogous to the respective results in the self-adjoint case \cite{PSS}, the proofs cannot be carried over from the self-adjoint case via standard transformations relating unitary and self-adjoint operators and we provide an independent treatment for the case of unitaries.

Throughout the paper, $\sH$ denotes a separable Hilbert space and
$S^n(\mathcal{B}(\sH))$ (or merely $S^n$) the $n$th Schatten-von Neumann ideal on $\sH$, that is,\[S^n(\mathcal{B}(\sH))=\big\{A\in \mathcal{B}(\sH):\, \|A\|_n:=\tr\big(|A|^n\big)^{1/n}<\infty\big\},\] where $\tr$ is the standard trace.

\section{Proof of the principal estimates.}
\label{sec2}


The following differentiation formulas for monomials of contractions can be established directly by definition of the G\^{a}teaux derivative (with convergence in the operator norm) and the method of mathematical induction.
\begin{lemma}
\label{thm:PolDer}
Let $U_0$ and $V$ be elements in $\mathcal{B}(\sH)$ and let $n, k\in \N$. Then,
\begin{align}
\label{PolDer}
\frac{d^n}{dt^n}\big(U_t^{k}\big)=
\begin{cases}
n!\sum\limits_{\substack{k_0,k_1,\dots,k_n\geq 0\\k_0+k_1+\dots +k_n=k-n}}U_t^{k_0}VU_t^{k_1}V\dots VU_t^{k_n},&\;\text{if } n\leq k,\\
0,&\;\text{if } n> k.
\end{cases}
\end{align}
\end{lemma}

\begin{lemma}
\label{thm:TrDer}
Assume Notations \ref{hyp} \eqref{hyp1} and assume $V\in S^n$. Then, for $f$ a polynomial and $t\in [0,1]$, 
\begin{equation}
\label{eq:TrDer}
\tr\left(\frac{d^n}{dt^n}f(U_t)\right)=\tr\left(\frac{d^{n-1}}{dt^{n-1}}f'(U_t)V\right).
\end{equation}
\end{lemma}

\begin{proof}
It is sufficient to prove the lemma for monomial~$f(x) = x^m$, $m
  \in \N$.  From Lemma \ref{thm:PolDer}, $$ \frac {d^n}{dt^n}
  f(U_t) = n!\, \sum\limits_{\substack{k_0,k_1,\dots,k_n\geq 0\\k_0+k_1+\dots +k_n=m-n}} U_t^{k_0} V \cdot
  \ldots \cdot V U_t^{k_n}. $$ Applying trace and using its cyclicity,
  we further have $$ \tr \left( \frac {d^n}{dt^n} f(U_t) \right) =
  n!\, \tr \left( \sum\limits_{\substack{k_0,k_1,\dots,k_n\geq 0\\k_0+k_1+\dots +k_{n-1}=m-n}} \left( k_0 + 1
    \right)\, U_t^{k_0} V \cdot \ldots \cdot V U_{t}^{k_{n-1}} V
  \right). $$ Using cyclicity again and reindexing, we also have
  that
  \begin{multline*}
    \tr \left( \frac {d^n}{dt^n} f(U_t) \right) = n!\, \tr
    \left( \sum\limits_{\substack{k_0,k_1,\dots,k_n\geq 0\\k_0+k_1+\dots +k_{n-1}=m-n}} \left( k_j + 1
      \right)\, U_t^{k_0} V \cdot \ldots \cdot V U_{t}^{k_{n-1}} V
    \right),\\ \text{where}\ \ j = 0, 1, \ldots, n-1.
  \end{multline*}
  Taking the sum over~$j=0, 1, \ldots, n-1$, we arrive at
  \begin{equation*}
    \tr \left( \frac {d^n}{dt^n} f(U_t) \right) = m\, (n-1)!\, \tr
    \left( \sum\limits_{\substack{k_0,k_1,\dots,k_n\geq 0\\k_0+k_1+\dots +k_{n-1}=m-n}} U_t^{k_0} V \cdot \ldots
      \cdot U_t^{k_{n-1}} V \right).
  \end{equation*}
  Hence, using Lemma \ref{thm:PolDer} again, we obtain
  $$  \tr \left( \frac {d^n}{dt^n} f(U_t) \right)  = \tr \left( V\, \frac
    {d^{n-1}}{dt^{n-1}} f'(U_t) \right). $$
\end{proof}



The proof of Theorem \ref{thm:MainEst} is simplified by the two following lemmas and will rely on Theorems \ref{thm:IndBase} and \ref{thm:IndStep}.

\begin{lemma}
\label{tou}
Let $n\in\N$, $\alpha>n$, and $V\in S^\alpha$. If the estimate \eqref{eq:MainEst} with $t_0=0$
holds for any unitary $U_0$ and contraction $U_1$ on every separable Hilbert space $\sK$ such that $V=U_1-U_0\in S^\alpha(\mathcal{B}(\sK))$, 
then it holds with any $t_0\in[0,1]$
for all contractions $\mU_0$ and $\mU_1$ on $\sH$ such that $\mathcal{V}=\mU_1-\mU_0 \in S^\alpha(\mathcal{B}(\sH))$.
\end{lemma}

\begin{proof}Fix $t_0\in [0,1]$.
Let $U_{t_0}$ be the minimal unitary dilation (which is unique up to an isomorphism) of $\mU_{t_0}=\mU_0+t_0\mathcal{V}$ defined on the space $\sK\supset\sH$. Then, $f\mapsto f(\mU_{t_0})=P_\sH f(U_{t_0})\big|_\sH$ for every polynomial $f$ \cite[Theorem 4.2]{SNF}. From \eqref{PolDer} we derive
\begin{align}
\label{pd2}
\frac{d^{n}}{dt^{n}}\big(f(\mU_0+t\mathcal{V})\big)\big|_{t=t_0}
=P_\sH \left(\frac{d^{n}}{dt^{n}}f\big(U_{t_0}+tP_\sH \mathcal{V} P_\sH\big)\big|_{t=0}\right)\bigg|_\sH.
\end{align}
(Dilations for more general multiple operator integrals were performed in \cite[Lemma 3.3]{PellerContr}. The proof above owes to the approach of \cite{PellerContr}.)

Let $\sH^\perp$ denote the orthogonal complement of $\sH$ in $\sK$. We have
\[V:=P_\sH \mathcal{V} P_\sH\in S^\alpha (\mathcal{B}(\sH)\oplus\mathcal{B}(\sH^\perp))\] and $\|V\|_\alpha=\|\mathcal{V}\|_\alpha$. Since, by the assumption,
\[\frac{d^{n}}{dt^{n}}f\big(U_{t_0}+tV\big)\big|_{t=0}\in S^{\frac{\alpha}{n}}(\mathcal{B}(\sH)\oplus\mathcal{B}(\sH^\perp))\] and \eqref{eq:MainEst} holds for the unitary $U_{t_0}$ and perturbation $V$, we deduce from \eqref{pd2} that
$\frac{d^{n}}{dt^{n}}\big(f(\mU_0+t\mathcal{V})\big)\big|_{t=t_0}\in S^{\frac{\alpha}{n}}(\mathcal{B}(\sH))$ and \eqref{eq:MainEst} holds for $\mU_0$ and $\mathcal{V}$.
\end{proof}

\begin{lemma}
\label{tofinite}
Assume Notations \ref{hyp} \eqref{hyp1} and assume that $V\in S^\alpha$, with $\alpha>n$.
If the estimate \eqref{eq:MainEst} holds for every unitary $U_0$ whose spectrum is a finite set, then it holds for an arbitrary unitary $U_0$.
\end{lemma}

\begin{proof}
Let
\[U_{0,N}:=\sum_{j=0}^{N-1}e^{2\pi\i j/N}E\big(\big[j/N,(j+1)/N\big)\big),\] where $E$ is the spectral measure of $U_0$. Then, $\|U_0^k-U_{0,N}^k\|\leq\frac{2\pi k}{N}$, with $k\in\N$. Since the left hand side of \eqref{eq:MainEst} depends continuously on $U_0$ in the operator norm (see \eqref{PolDer}), passing to the limit as $N\rightarrow\infty$ completes the proof. \end{proof}

Therefore, we can and shall assume in the proofs below that $U_0$ is a unitary whose spectrum is a finite set. Also note that it is enough to establish \eqref{eq:MainEst} for a subsequence of $\{U_{0,N}\}_{N\in\N}$ (rather than the whole sequence).

\begin{definition}
\label{MOI-def}
Let $n,N\in\N$. Denote $$z_j:=e^{2\pi\i j/N}$$ and given a spectral measure $E$ on $\T$, denote $$E_j:=E(z_j),\quad j=0,\dots,N-1.$$ Let $1\leq\alpha_i\leq\infty$, for $i=1,\dots,n$, and $1\leq\alpha\leq\infty$ be such that $\frac{1}{\alpha_1}+\cdots+\frac{1}{\alpha_n}=\frac{1}{\alpha}$. For $\phi$ a bounded Borel function on $\T^{n+1}$ and $B$ a Borel subset of $T^{n+1}$, we define the mapping \[(x_1,\dots,x_n)\mapsto T^B_{\phi}(x_1,\dots,x_n)\] on $S^{\alpha_1} \times \ldots \times S^{\alpha_n}$ with values in\footnote{We shall frequently omit writing the image space of $T_\phi$.} $S^\alpha$, called a multiple operator integral with symbol $\phi$, by
\begin{align}
\label{TphiB}
T^B_\phi(x_1,\dots,x_n):=\sum\limits_{(z_{j_0},\dots,z_{j_n})\in B}
\phi(z_{j_0},\dots,z_{j_n})E_{j_0}x_1 E_{j_1} x_2\dots x_n E_{j_n}.
\end{align}
We also use the shortcut $T_\phi:=T_\phi^{\T^{n+1}}$.
\end{definition}

Note that if $B$ and $C$ are disjoint Borel subsets of $\T^{n+1}$, we have the additivity of the multiple operator integral over the region
\[T^{B\cup C}_\phi = T^B_\phi+T^C_\phi\] and if $\phi,\psi$ are bounded Borel functions on $\T^{n+1}$, we have the additivity over the symbol
\[T^B_{\phi+\psi}=T^B_\phi+T^B_\psi.\]

We recall that the divided difference of the zeroth order~$f^{[0]}$ is
the function~$f$ itself. Let~$\lambda_0, \lambda_1, \ldots \in \R$
and let~$f \in C^n$. The divided difference $f^{[n]}$ of order~$n$ is
defined recursively by
\begin{align*}
  f^{[n]} \left( \lambda_0, \lambda_1, \lambda_2, \ldots, \lambda_n \right) =
  \begin{cases}\frac
    { f^{[n-1]} (\lambda_0, \lambda_2, \ldots, \lambda_n) - f^{[n-1]}(\lambda_1,
      \lambda_2, \ldots, \lambda_n)}{\lambda_0 - \lambda_1}, & \text{if~$\lambda_0
      \neq \lambda_1$}, \\ \frac {d}{d\lambda_1} f^{[n-1]} (\lambda_1,
    \lambda_2, \ldots, \lambda_n), & \text{if~$\lambda_0=\lambda_1$}.
  \end{cases}
\end{align*}
If $f(\lambda)=\lambda^k$, then
\[f^{[n]} \left( \lambda_0, \lambda_1, \ldots, \lambda_n \right)=
\sum\limits_{\substack{k_0,k_1,\dots,k_n\geq 0\\k_0+k_1+\dots +k_n=k-n}}
\lambda_0^{k_0}\lambda_1^{k_1}\cdots\lambda_n^{k_n}\]

\begin{lemma}
\label{der=dd}
Assume Notations \ref{hyp} \eqref{hyp1}.
Suppose $U_0$ is unitary and the spectrum of $U_0$ is concentrated at the points $\{z_j\}_{j=0}^{N-1}$.
Let $E$ be the spectral measure of $U_0$, let $f$ be a polynomial, and $V\in \mathcal{B}(\sH)$. Then,
\[\frac{d^n}{dt^n}\big(f(U_0+tV)\big)\big|_{t=0}=n!T_{f^{[n]}}(\underbrace{V, \ldots,
        V}_{n\text{ \rm times}}).\]
\end{lemma}

\begin{proof}
It is enough to prove the lemma for an arbitrary monomial $f(t)=t^k$.
By the spectral theorem applied to $U_t$ with $t=0$ on the right hand side of \eqref{PolDer}, we obtain the needed formula
\begin{align*}
\frac{d^n}{dt^n}\big(U_t^k\big|_{t=0}\big)
&=n!\sum_{j_0,j_1,\dots,j_n=0}^{N-1}
\sum\limits_{\substack{k_0,k_1,\dots,k_n\geq 0\\k_0+k_1+\dots +k_n=k-n}}
z_{j_0}^{k_0}z_{j_1}^{k_1}\dots z_{j_n}^{k_n}E_{j_0}VE_{j_1}V\dots VE_{j_n}\\
&=n!\sum_{j_0,j_1,\dots,j_n=0}^{N-1}f^{[n]}(z_{j_0},z_{j_1},\dots,z_{j_n})E_{j_0}VE_{j_1}V\dots VE_{j_n}.
\end{align*}
\end{proof}

Throughout the paper, we shall frequently use the following algebraic
properties of the mapping~$\phi \mapsto T_\phi^B$ built over a discrete measure of a unitary operator, whose self-adjoint counterpart was established in \cite[Lemma 3.2]{PSS}.

\begin{lemma}
  \label{MOI-algebra}
  Let~$1 \leq \alpha,\alpha_i \leq \infty$, for $1\leq i \leq n$, be such that
  $0 \leq \frac{1}{\alpha}=\frac{1}{\alpha_1}+\ldots+\frac1{\alpha_n}\leq 1$.
  Let~$x_i \in S^{\alpha_i}$, $1\leq i \leq n$. Let $B$ be a Borel subset of $\T^{n+1}$.

\begin{enumerate}
\item Let~$\phi: \T^{n + 1} \mapsto \C$ be a bounded Borel function and
  let the transformation~$T_\phi^B: S^{\alpha_1} \times \ldots \times S^{\alpha_n}\mapsto S^\alpha$ be bounded.  If $$ \bar \phi(\lambda_0, \lambda_1, \ldots,
  \lambda_n) := \overline{\phi (\lambda_n, \lambda_{n-1}, \ldots,
    \lambda_0)}, $$
\[\bar B:=\{(\lambda_0,\lambda_1,\dots,\lambda_n): (\lambda_n,\lambda_{n-1},\dots,\lambda_0)\in B\}, \]
then~$T_{\bar \phi}^{\bar B}: S^{\alpha_1} \times \ldots \times S^{\alpha_n}\mapsto S^\alpha$ is bounded and $$ \left\| T_{\phi}^B \right\| = \left\| T_{\bar \phi}^{\bar B} \right\|. $$ \label{MOI-A-involution}
\item Assume, in addition, that $1 \leq\alpha_0\leq\infty$ and
  $\frac{1}{\alpha_0}+\ldots+\frac1{\alpha_n}=1$.  Let~$\phi: \T^{n +
    1} \mapsto \C$ be a bounded Borel function. Assume that $T_\phi^B$ is bounded on $S^{\alpha_1} \times \ldots \times S^{\alpha_n}$. 
  Define $$ \phi^*(\lambda_{n}, \lambda_0, \ldots,
  \lambda_{n-1}) := \phi \left( \lambda_0, \ldots, \lambda_{n-1},
    \lambda_n \right), $$
 \[B^*:=\{(\lambda_0,\lambda_1,\dots,\lambda_n): (\lambda_n,\lambda_0,\dots,\lambda_{n-1})\in B\}.\]
Then,~$T^{B^*}_{\phi^*}$ 
    is bounded  on~$S^{\alpha_0} \times \ldots \times S^{\alpha_{n-1}}$ and
    $$ \tr\left( x_0 T^B_\phi (x_1, \ldots, x_n) \right) = \tr \left(
    T^{B^*}_{\phi^*} \left( x_0, \ldots , x_{n-1} \right) x_n
  \right). $$ \label{MOI-A-duality}
  \item Let~$\phi_1: \T^{k + 1} \mapsto \C$ and~$\phi_2 : \T^{n - k
      + 1} \mapsto \C$ be bounded Borel functions and let $B_1$ and $B_2$ be Borel subsets of $\T^{k+1}$ and $\T^{n-k+1}$, respectively. Suppose that the operators~$T^{B_1}_{\phi_1}$ and~$T^{B_2}_{\phi_2}$
    are bounded on~$S^{\alpha_1} \times \ldots \times S^{\alpha_k}$ and~$S^{\alpha_{k +
        1}} \times \ldots \times S^{\alpha_n}$, respectively.  If $$
    \psi(\lambda_0, \ldots, \lambda_n) := \phi_1 \left( \lambda_0,
      \ldots, \lambda_k \right) \cdot \phi_2 \left( \lambda_k, \ldots,
      \lambda_n \right), $$
  \[\tilde B:=\{(\lambda_0,\dots,\lambda_k,\dots,\lambda_n):
   (\lambda_0,\dots,\lambda_k)\in B_1, (\lambda_k,\dots,\lambda_n)\in B_2\},\]
  then
    the operator~$T^{\tilde B}_{\psi}:S^{\alpha_1}\times \ldots \times S^{\alpha_n}\mapsto S^\alpha$ 
    is bounded and $$ T^{\tilde B}_{\psi} \left( x_1,
      \ldots, x_n \right) = T^{B_1}_{\phi_1} \left( x_1, \ldots, x_k \right)
    \cdot T^{B_2}_{\phi_2} \left( x_{k + 1}, \ldots, x_n
    \right). $$ \label{MOI-A-product}
  \item Let~$\phi_1: \T^{k + 1} \mapsto \C$ and~$\phi_2 : \T^{n - k
      + 2} \mapsto \C$ be bounded Borel functions and let $B_1$ and $B_2$ be bounded Borel subsets of $\T^{k+1}$ and $\T^{n-k+2}$, respectively. Suppose
    that~$T^{B_1}_{\phi_1}$ and~$T^{B_2}_{\phi_2}$ 
    are bounded on~$S^{\alpha_1} \times \ldots \times S^{\alpha_k}$
    and~$S^{\alpha_0} \times S^{\alpha_{k+1}} \times \ldots \times
    S^{\alpha_n}$, respectively, where
    $\frac{1}{\alpha_0}=\frac{1}{\alpha_1}+\ldots+\frac1{\alpha_k}$. If $$
    \psi (\lambda_0, \ldots, \lambda_n): = \phi_1 \left( \lambda_0,
      \ldots, \lambda_k \right) \cdot \phi_2 \left(\lambda_0,
      \lambda_{k}, \ldots, \lambda_n \right), $$
   \[\tilde B:=\{(\lambda_0,\dots,\lambda_k,\dots,\lambda_n):
   (\lambda_0,\dots,\lambda_k)\in B_1, (\lambda_0,\lambda_k,\dots,\lambda_n)\in B_2\},\]
  then
    the operator~$T^{\tilde B}_{\psi}:S^{\alpha_1}\times\ldots\times S^{\alpha_n}\mapsto S^\alpha$ is bounded and
    $$T_{\psi}^{\tilde B}\left( x_1, \ldots, x_n \right) =
    T^{B_2}_{\phi_2} \left(T^{B_1}_{\phi_1}(x_1, \ldots, x_k), x_{k + 1}, \ldots, x_n
    \right). $$ \label{MOI-A-composition}
  \end{enumerate}
\end{lemma}

\begin{proof}
Assertion \eqref{MOI-A-involution} can be established by taking the adjoint in \eqref{TphiB}, \eqref{MOI-A-duality} follows from the cyclicity of the trace, and \eqref{MOI-A-product} and \eqref{MOI-A-composition} can be verified by comparison of the multiple operator integrals that appear on both sides of the equalities.
\end{proof}

The main estimate \eqref{eq:MainEst} has an antecedent in the self-adjoint case \cite{PSS}; however, many subtle details in the proof need to be changed. The case of unitaries is technically more involved than the case of self-adjoints, and to compensate for the increasing complexity, we make a ``shortcut" through use of classical complex analysis and some results of \cite{KPSS} (which are based on multidimensional harmonic analysis).

The estimate \eqref{eq:MainEst} is proved by induction on $n$; the base of induction is established in the following theorem.

\begin{theorem}
\label{thm:IndBase}
Let $h$ be a polynomial, $m\in \N\cup\{0\}$, and $\lambda,\mu\in\overline\D$. The double operator integral with the symbol
\begin{align}
\label{phimh}
\phi_{h,m}(\lambda,\mu):=\int_0^1 t^m h(\lambda+(\mu-\lambda)t)\,dt
\end{align} is bounded on $S^\alpha$, $1<\alpha<\infty$, and
\begin{align}
\label{idest}
\|T_{\phi_{h,m}}\|_\alpha\leq c_{\alpha,m}\|h\|_\infty,
\end{align}
where $\|T_{\phi_{h,m}}\|_\alpha$ is the norm of the operator $T_{\phi_{h,m}}: S^\alpha\mapsto S^\alpha$.
\end{theorem}

\begin{remark}
If $m=0$ and $h=f'$, then $\phi_{h,m}=f^{[1]}$.
\end{remark}

To prove the estimate of Theorem \ref{thm:IndBase}, we utilize the following decomposition, which has a complex analytic proof (as distinct from its counterpart \cite[Lemma 5.7]{PSS} in the self-adjoint case).

\begin{lemma}
\label{thm:BaseDecomp}
Let $m\in \N\cup\{0\}$.
For $\lambda, \xi, \mu \in \T$, with $\lambda\neq\mu$, and $h$ a polynomial,
\begin{align*}
\phi_{h,m}(\lambda,\mu)&=
\left(\frac{\xi-\lambda}{\mu-\lambda}\right)^{m+1}\phi_{h,m}(\lambda,\xi)
+\left(\frac{\mu-\xi}{\mu-\lambda}\right)^{m+1}\phi_{h,m}(\xi,\mu)\\
&\quad + \sum_{k=0}^{m-1}C_m^k\left(\frac{\mu-\xi}{\mu-\lambda}\right)^{k+1}
\left(\frac{\xi-\lambda}{\mu-\lambda}\right)^{m-k}\phi_{h,k}(\xi,\mu),
\end{align*}
where the third sum is not present if $m=0$.
\end{lemma}

\begin{proof}
Let $[\xi,\mu]$ denote the segment beginning at point $\xi\in\C$ and ending at point $\mu\in\C$.

In the first integral, we make change of variables $\omega(t)=\lambda+(\mu-\lambda)t$. The function $\omega$ attains its values in $\overline\D$. We note that $t=\frac{\omega-\lambda}{\mu-\lambda}$ and $dt=\frac{1}{\mu-\lambda}\,d\omega$ and, hence,
\begin{align*}
\int_0^1 t^m h(\lambda+(\mu-\lambda)t)\,dt=\int_{[\lambda,\mu]} \left(\frac{\omega-\lambda}{\mu-\lambda}\right)^m h(\omega)\frac{d\omega}{\mu-\lambda}.
\end{align*}
By the Cauchy integral theorem,
\begin{align*}
\int_{[\lambda,\xi]+[\xi,\mu]+[\mu,\lambda]} \left(\frac{\omega-\lambda}{\mu-\lambda}\right)^m h(\omega)\,d\omega
=0,
\end{align*}
which, by additivity of the integral over the region of integration, implies
\begin{align*}
&\int_{[\lambda,\mu]} \left(\frac{\omega-\lambda}{\mu-\lambda}\right)^m h(\omega)\frac{d\omega}{\mu-\lambda}\\
&\quad=\int_{[\lambda,\xi]}\left(\frac{\omega-\lambda}{\mu-\lambda}\right)^m h(\omega)\frac{d\omega}{\mu-\lambda}
+\int_{[\xi,\mu]}\left(\frac{\omega-\lambda}{\mu-\lambda}\right)^m h(\omega)\frac{d\omega}{\mu-\lambda}.
\end{align*}
Using the straightforward decompositions
\begin{align*}
&\frac{\omega-\lambda}{\mu-\lambda}=\left(\frac{\omega-\lambda}{\xi-\lambda}\right)
\left(\frac{\xi-\lambda}{\mu-\lambda}\right),\\
&\left(\frac{\omega-\lambda}{\mu-\lambda}\right)^m=
\left(\frac{\omega-\xi}{\mu-\lambda}+\frac{\xi-\lambda}{\mu-\lambda}\right)^m
=\sum_{k=0}^{m-1}C_m^k \left(\frac{\omega-\xi}{\mu-\xi}\right)^k\left(\frac{\mu-\xi}{\mu-\lambda}\right)^k
\left(\frac{\xi-\lambda}{\mu-\lambda}\right)^{m-k},
\end{align*}
we derive
\begin{align*}
&\int_0^1 t^m h(\lambda+(\mu-\lambda)t)\,dt\\
&\quad=\left(\frac{\xi-\lambda}{\mu-\lambda}\right)^{m+1}\int_0^1 t^m h(\lambda+(\xi-\lambda)t)\,dt+\left(\frac{\mu-\xi}{\mu-\lambda}\right)^{m+1}\int_0^1 t^m h(\xi+(\mu-\xi)t)\,dt\\
&\quad\quad+\sum_{k=0}^{m-1}C_m^k\left(\frac{\mu-\xi}{\mu-\lambda}\right)^{k+1}
\left(\frac{\xi-\lambda}{\mu-\lambda}\right)^{m-k}\int_0^1 t^k h(\xi+(\mu-\xi)t)\,dt.
\end{align*}
\end{proof}

In the proof of Theorem \ref{thm:IndBase}, we shall need to factorize the double operator integral according to the decomposition of the symbol $\phi_{h,m}$ derived in Lemma \ref{thm:FactorDecomp} below.

Recall the following useful representation for positive fractions.

\begin{lemma}(\cite[Lemma 6]{PS-Acta})
\label{thm:g}
Let $\delta\in\R_+$.
There exists $g_\delta: \R\mapsto\C$ such that $\int_\R |s|^k|g_\delta(s)|\,ds<\infty$, $k\geq 0$, and such that for all $\lambda,\mu >0$ with $0\leq\frac\lambda\mu\leq \delta$,
\[\frac\lambda\mu=\int_\R g_\delta(s)\lambda^{\i s}\mu^{-\i s}\,ds.\]
\end{lemma}

\begin{lemma}
\label{thm:FactorDecomp}
Let $w\in\T$, $\delta\in\R_+$, $i,j\in\{0,\dots,N-1\}$, and let $g_\delta$ be the function from Lemma \ref{thm:g}. If $\frac{|z_i-w|}{|z_j-z_i|}\leq \delta$ and $\frac{|z_j-w|}{|z_j-z_i|}\leq \delta$, then for $m\in\N\cup\{0\}$,
\begin{align*}
&\phi_{h,m}(z_i,z_j)\\
&\;=\int_\R g_\delta(s)\left(\frac{|z_j-z_i|}{z_j-z_i}\right)^{m+1}
\left(\frac{z_i-w}{|z_i-w|}\right)^{m+1}\left(\frac{|z_i-w|}{|z_j-z_i|}\right)^{\i(m+1)s}\phi_{h,m}(z_i,w)\,ds
\displaybreak[2] \\
&\quad +\int_\R g_\delta(s)\left(\frac{|z_j-z_i|}{z_j-z_i}\right)^{m+1}
\left(\frac{z_j-w}{|z_j-w|}\right)^{m+1}\left(\frac{|z_j-w|}{|z_j-z_i|}\right)^{\i(m+1)s}\phi_{h,m}(w,z_j)\,ds
\displaybreak[2] \\
&\quad + \sum_{k=0}^{m-1}C_m^k\int_\R g_\delta(s)\left(\frac{|z_j-z_i|}{z_j-z_i}\right)^{m+1}
\left(\frac{|z_i-w|^{\i(m-k)s}|z_j-w|^{\i(k+1)s}}{|z_j-z_i|^{\i(m+1)s}}\right)\\
&\quad\quad\quad\quad\quad\quad\quad\quad\quad\quad
\cdot\left(\frac{z_j-w}{|z_j-w|}\right)^{k+1}
\left(\frac{z_i-w}{|z_i-w|}\right)^{m-k}\phi_{h,k}(w,z_j)\,ds.
\end{align*}
\end{lemma}

\begin{proof}
The result follows from the straightforward decomposition for complex numbers\footnote{If $z=0$, we define
$\frac{|z|}{z}:=0$ and $\frac{z}{|z|}:=0$.}
\begin{align*}
\frac{\mu-\xi}{\mu-\lambda}=
\left(\frac{|\mu-\lambda|}{\mu-\lambda}\right)\left(\frac{\mu-\xi}{|\mu-\xi|}\right)
\left(\frac{|\mu-\xi|}{|\mu-\lambda|}\right),\quad \mu\neq\lambda.
\end{align*}
and Lemmas \ref{thm:BaseDecomp} and \ref{thm:g}.
\end{proof}

Application of \cite[Theorem 3.4]{KPSS} to the functions
\[g_1(x_1,x_2)=\frac{x_1}{\sqrt{x_1^2+x_2^2}},\quad g_2(x_1,x_2)=\frac{x_2}{\sqrt{x_1^2+x_2^2}},\quad
g_3(x_1,x_2)=\big(\sqrt{x_1^2+x_2^2}\big)^{\i s}\] defined on $\R^2\setminus\{0\}$ and multiplicativity of the double operator integral from Lemma \ref{MOI-algebra} \eqref{MOI-A-composition} implies the following result, which will be frequently applied in the paper.

\begin{lemma}
\label{thm:KPSS}
Let $B$, $C$ be subsets\footnote{The sets $B$ and $C$ vary, but the estimates do not depend on the choice of $B$ and $C$, so they are omitted in the notation of the respective double operator integrals.} of $\{0,\dots,N-1\}$ and let $m\in\N$, $s\in\R$. For $x\in S^\alpha$, $1<\alpha<\infty$, denote
\begin{align*}
&\Upsilon_m(x):=\sum_{i\in B, j\in C}\left(\frac{z_i-z_j}{|z_i-z_j|}\right)^m E_i x E_j\\
&\Upsilon_{-m}(x):=\sum_{i\in B, j\in C}\left(\frac{|z_i-z_j|}{z_i-z_j}\right)^m E_i x E_j\\
&\Gamma_s(x):=\sum_{i\in B, j\in C}|z_i-z_j|^{\i s} E_i x E_j.
\end{align*}
Then, there are constants $c_\alpha$ and $c_{\alpha,m}$ such that
\begin{align*}
&\|\Upsilon_m(x)\|_\alpha\leq c_{\alpha,m}\|x\|_\alpha,\\
&\|\Upsilon_{-m}(x)\|_\alpha\leq c_{\alpha,m}\|x\|_\alpha,\\
&\|\Gamma_s(x)\|_\alpha\leq c_{\alpha}(1+|s|+|s|^2)\|x\|_\alpha.
\end{align*}
\end{lemma}

\begin{proof}[Proof of Theorem \ref{thm:IndBase}]
Denote
\[Q_k:=\left\{z\in\T: \arg(z)\in \left[\frac{2\pi k}{3},\frac{2\pi (k+1)}{3}\right)\right\},
\quad k=0,1,2,\]
\[D_{k_0,k_1}=\{(z_{j_0},z_{j_1}):\; z_{j_0},z_{j_1}\in Q_{k_0}\cup Q_{k_1}\},\quad k_0,k_1=0,1,2,\]
\[D_d=\{(z_{j_0},z_{j_1}):\, z_{j_0},z_{j_1}\in Q_k,\, k=0,1,2\}.\]
Since
\[T_{\phi_{h,m}}=T_{\phi_{h,m}}^{D_{0,1}}+T_{\phi_{h,m}}^{D_{1,2}}+T_{\phi_{h,m}}^{D_{0,2}}
-T_{\phi_{h,m}}^{D_d},\]
it is enough to prove the theorem separately for each of the summands. We shall demonstrate only the case of
$T_{\phi_{h,m}}^{D_{0,1}}$; the other cases can be handled completely analogously.

The estimate \eqref{idest} for $\alpha=2$ is well-known.
The boundedness of $T_{\phi_{h,m}}^{D_{0,1}}$ for $\alpha\neq 2$ is proved similarly to how it was done in \cite[Theorems 4.1 and 5.6]{PSS} in the self-adjoint case.

As a first step, we show that if $\alpha,\beta\in (2,\infty)$ and \[2^{-1}=\alpha^{-1}+\beta^{-1}\] and if $\|h\|_\infty\leq 1$ (and, hence, $\|\phi_{h,m}\|_\infty\leq 1$), then
\begin{align}
\label{sa0}
\big\|T_{\phi_{h,m}}^{D_{0,1}}\big\|_\alpha\leq c_{\alpha,m}\left(1+\big\|T_{\phi_{h,m}}^{D_{0,1}}\big\|_\beta\right),
\end{align} by showing
\begin{align}
\label{sa1}
\left|\tr\left(y T_{\phi_{h,m}}^{D_{0,1}}(x)\right)\right|\leq c_{\alpha,m}\left(1+\big\|T_{\phi_{h,m}}^{D_{0,1}}\big\|_\beta\right)\|x\|_\alpha\|y\|_{\alpha'},
\end{align} where $\alpha^{-1}+\alpha'^{-1}=1$. In the proof below we assume that
$\|x\|_\alpha=1$ and $\|y\|_{\alpha'}=1$.

Let $N$ from Definition \ref{MOI-def} be divisible by $3$.
To show \eqref{sa1}, recall that the triangular truncation is a bounded linear operator on $S^\alpha$, $1<\alpha<\infty$ (see, e.g., \cite{DDPS} or \cite{Volterra}). By standard techniques, one can see that $T_{\phi_{h,m}}^{A_0}$ is bounded on the diagonal set $A_0=\{(z_{j_0},z_{j_1})\in D_{0,1}:\; z_{j_0}=z_{j_1}\}$. (Details can be found on p. 383 of \cite{PS-Acta}. We will provide a more general argument in Lemma 2.18.) Thus, we can assume that $x$ is upper triangular and off-diagonal and $y$ is lower triangular with respect to the family of projections $\{E_j\}_{j=0}^{N/3-1}$ (as in Definition \ref{MOI-def}).

We can assume that $y$ is finite rank because the class of lower triangular finite rank operators is norm dense in the lower-triangular part of $S^{\alpha'}$. Given $\epsilon>0$, there is a factorization $y=ab$, where $a\in S^2$ and $b\in S^\beta$ are lower triangular and \[1\leq\|a\|_2\|b\|_\beta\leq 1+\epsilon\] (see, e.g., \cite[Theorem 8.3]{PiXu} and references cited therein). Therefore,
\[\tr\left(y T_{\phi_{h,m}}^{D_{0,1}}(x)\right)=\sum_{i\leq l\leq j,\; i\neq j\atop i,j,l=0,\dots,N/3-1\;}
\phi_{h,m}(z_i,z_j)\tr\big(E_j a E_l b E_i x E_j\big).\]
Note that for $i,j,l$ as in the summation above and $\delta=2/\sqrt{3}$, we have
$\frac{|z_i-z_l|}{|z_j-z_i|}\leq \delta$ and $\frac{|z_j-z_l|}{|z_j-z_i|}\leq \delta$.
The algebraic properties of the multiple operator integrals stated in Lemma \ref{MOI-algebra} and the decomposition of the function $\phi_{h,m}$ from Lemma \ref{thm:FactorDecomp} with $\omega=z_l$ imply
\begin{align}
\label{5.15}
&\tr\left(y T_{\phi_{h,m}}^{D_{0,1}}(x)\right)\\
\nonumber
&\;=\int_\R g_\delta(s)\tr\left(a\cdot\Upsilon_{m+1}\left(\Gamma_{s(m+1)}\big(T_{\bar\phi_{h,m}}(b)\big)\right)\cdot \Upsilon_{-(m+1)}\big(\Gamma_{-s(m+1)}(x)\big)\right)\,ds \displaybreak[2] \\
\nonumber
&\quad + \int_\R g_\delta(s)\tr\left(\Upsilon_{m+1}\left(\Gamma_{s(m+1)}\big(T_{\bar\phi_{h,m}}(a)\big)\right)\cdot b\cdot \Upsilon_{-(m+1)}\big(\Gamma_{-s(m+1)}(x)\big)\right)\,ds \displaybreak[2] \\
\nonumber
&\quad + \sum_{k=0}^{m-1}C_m^k\int_\R g_\delta(s)\tr\bigg(T_{\bar\phi_{h,k}}\left(\Upsilon_{k+1}\big(\Gamma_{s(k+1)}(a)\big)\right)
\cdot\Upsilon_{1}\big(\Gamma_{s(m-k)}(b)\big)\\
\nonumber
&\quad\quad\quad\quad\quad\quad\quad\quad\quad\quad\quad\quad\quad\quad\quad\quad\quad\quad\quad\quad
\cdot\Upsilon_{-(m+1)}\big(\Gamma_{-s(m+1)}(x)\big)\bigg)\,ds,
\end{align} where $\bar\phi_{h,m}(\lambda,\mu)=\phi_{h,m}(\mu,\lambda)$.
By properties of the double operator integral on $S^2$ and Lemma \ref{thm:KPSS},
\[\left\|T_{\bar\phi_{h,k}}\left(\Upsilon_{k+1}\big(\Gamma_{s(k+1)}(a)\big)\right)\right\|_2
\leq\|\phi_{h,k}\|_\infty\big\|\Upsilon_{k+1}\big(\Gamma_{s(k+1)}(a)\big)\big\|_2\leq\|a\|_2.\]
Next, we apply Lemma \ref{thm:KPSS} and $\|T_{\phi_{h,m}}\|=\|T_{\bar\phi_{h,m}}\|$ (see Lemma \ref{MOI-algebra} \eqref{MOI-A-involution}) to derive
\begin{align*}
&\bigg|\tr\bigg(T_{\bar\phi_{h,k}}\left(\Upsilon_{k+1}\big(\Gamma_{s(k+1)}(a)\big)\right)
\cdot\Upsilon_{1}\big(\Gamma_{s(m-k)}(b)\big)
\cdot\Upsilon_{-(m+1)}\big(\Gamma_{-s(m+1)}(x)\big)\bigg)\bigg|\\
&\quad \leq c_\alpha\|a\|_2\left\|\Upsilon_{1}\big(\Gamma_{s(m-k)}(b)\big)\right\|_\beta
\left\|\Upsilon_{-(m+1)}\big(\Gamma_{-s(m+1)}(x)\big)\right\|_\alpha\\
&\quad \leq c_{\alpha,m}(1+|s(m+1)|+|s(m+1)|^2)^2(1+\epsilon)
\end{align*}
as well as
\begin{align*}
&\bigg|\tr\left(a\cdot\Upsilon_{m+1}\left(\Gamma_{s(m+1)}\big(T_{\bar\phi_{h,m}}(b)\big)\right)\cdot \Upsilon_{-(m+1)}\big(\Gamma_{-s(m+1)}(x)\big)\right)\bigg|\\
&\quad \leq \|a\|_2\left\|\Upsilon_{m+1}\left(\Gamma_{s(m+1)}\big(T_{\bar\phi_{h,m}}(b)\big)\right)\right\|_\beta
\left\|\Upsilon_{-(m+1)}\big(\Gamma_{-s(m+1)}(x)\big)\right\|_\alpha\\
&\quad \leq c_{\alpha,m}(1+|s(m+1)|+|s(m+1)|^2)^2\|T_{\phi_{h,m}}\|_\beta(1+\epsilon)
\end{align*}
and
\begin{align*}
&\left|\tr\left(\Upsilon_{m+1}\left(\Gamma_{s(m+1)}\big(T_{\bar\phi_{h,m}}(a)\big)\right)\cdot b\cdot \Upsilon_{-(m+1)}\big(\Gamma_{-s(m+1)}(x)\big)\right)\right|\\
&\quad \leq c_{\alpha,m}(1+|s(m+1)|+|s(m+1)|^2)^2(1+\epsilon).
\end{align*}
By letting $\epsilon \rightarrow 0$, then applying the triangle inequality and just derived inequalities to \eqref{5.15}, we arrive at \eqref{sa1} and, hence, at \eqref{sa0}.

For $\alpha>2$, we derive $\big\|T_{\phi_{h,m}}^{D_{0,1}}\big\|_\alpha\leq c_{\alpha,m}$ from \eqref{sa0} by interpolation. Fix $\alpha>4$ and fix $\beta\in(2,\alpha)$.
There is $\theta\in (0,1)$ such that
\[\frac{1}{\beta}=\frac{1-\theta}{2}+\frac{\theta}{\alpha}.\]
By the complex interpolation method \cite{BL},
\[\big\|T_{\phi_{h,m}}^{D_{0,1}}\big\|_\beta\leq \big\|T_{\phi_{h,m}}^{D_{0,1}}\big\|_2^{1-\theta}
\big\|T_{\phi_{h,m}}^{D_{0,1}}\big\|_\alpha^\theta\leq \big\|T_{\phi_{h,m}}^{D_{0,1}}\big\|_\alpha^\theta.\]
Combining the latter inequality with \eqref{sa0} gives
\[\big\|T_{\phi_{h,m}}^{D_{0,1}}\big\|_\alpha\leq c_{\alpha,m}\bigg(1+\big\|T_{\phi_{h,m}}^{D_{0,1}}\big\|_\alpha^\theta\bigg),\quad 0<\theta<1,\]
which implies the estimate \eqref{idest} for $\alpha>4$. By duality and Lemma \ref{MOI-algebra} \eqref{MOI-A-duality}, we obtain this estimate for $1<\alpha<2$. Applying the interpolation again completes the proof of the theorem for all $\alpha \in (1,\infty)$.
\end{proof}

To make an inductive reduction to the lower order case, we need decompositions for functions more general than $\phi_{h,m}$.

For $h$ a polynomial, $m,k\in \N\cup\{0\}$, $\{\lambda_j\}_{j=0}^n\subset \T$, denote
\begin{align}
\label{phinhmk}
&\phi_{n,h,m,k}(\lambda_0,\dots,\lambda_n)\\
\nonumber
&\; :=\int_0^1\int_0^{t_n}\dots\int_0^{t_4}\int_0^{t_3}
\int_0^t t^m s^k h\big(\lambda_n+(\lambda_{n-1}-\lambda_n)t_n+\dots\\
\nonumber
&\quad\quad\quad
+(\lambda_2-\lambda_3)t_3+(\lambda_1-\lambda_2)t+(\lambda_0-\lambda_1)s\big)\,ds\, dt\, dt_3 \dots dt_{n-1}\, dt_n.
\end{align}
By a standard property of the divided difference \cite[Formula (7.12)]{DVL},
\[\phi_{n,f^{(n)},0,0}(\lambda_0,\dots,\lambda_n)=f^{[n]}(\lambda_0,\dots,\lambda_n).\]
In case of three variables, \eqref{phinhmk} should be understood as
\begin{align*}
\phi_{2,h,m,k}(\lambda_0,\lambda_1,\lambda_2)=
\int_0^1\int_0^t t^m s^k h\big(\lambda_2+(\lambda_1-\lambda_2)t+(\lambda_0-\lambda_1)s\big)\,ds\, dt
\end{align*}
and in case of two variables as
\begin{align*}
\phi_{1,h,m,k}(\lambda_0,\lambda_1)=
\int_0^1 t^k h\big(\lambda_1+(\lambda_0-\lambda_1)t\big)\, dt=\phi_{h,k}(\lambda_0,\lambda_1).
\end{align*}
Below, we reduce the functions $\phi_{n,h,m,k}$ to the same type of functions of the previous order first for $n=2$ (see Lemmas \ref{thm:tmh} and \ref{thm:tkh}) and then for $n>2$ (see Lemma \ref{thm:tmkh}).

Denote $u(t,s):=\kappa\xi+(\lambda-\xi)t+(\mu-\lambda)s$.

The following two lemmas have assertions similar to the one in \cite[Lemma 5.9]{PSS}; however, the proof of the latter does not extend to the complex plane. The main ingredient of the new method is the usage of Green's theorem.

\begin{lemma}
\label{thm:tmh}
Let $\kappa \in (0,1]$, $\lambda,\xi,\mu\in\T$, with $\lambda\neq\mu$, and 
let $h$ be a polynomial. Then,
\begin{align}
\label{eq:tmh}
&\int_0^\kappa\int_0^t t^{m-1}h(\kappa\xi+(\mu-\xi)t+(\lambda-\mu)s)\,ds\,dt\\
\nonumber
&=\frac1m\left(\frac{\xi-\lambda}{\mu-\lambda}\right)\int_0^\kappa (\kappa^m-t^m)h(\kappa\xi+(\lambda-\xi)t)\,dt\\
\nonumber
&\quad+\frac1m\left(\frac{\mu-\xi}{\mu-\lambda}\right)\int_0^\kappa (\kappa^m-t^m)h(\kappa\xi+(\mu-\xi)t)\,dt.
\end{align}
\end{lemma}

\begin{proof}
It is enough to prove the lemma for $\kappa=1$ and all $\lambda,\xi,\mu\in\overline\D$, with $|\lambda|=|\xi|=|\mu|$ and $\lambda\neq\mu$, then make the substitution $\lambda=\kappa\tilde\lambda$, $\xi=\kappa\tilde\xi$, and $\mu=\kappa\tilde\mu$, $\tilde t=\kappa t$, $\tilde s=\kappa s$ and derive the formula for $\tilde\lambda$, $\tilde\xi$, and $\tilde\mu$.

Note that
\begin{align}
\label{inteq}
\nonumber
&\int_0^1\int_0^t t^{m-1}h(\xi+(\mu-\xi)t+(\lambda-\mu)s)\,ds\,dt\\
&\;=\int_0^1\int_0^t t^{m-1}h(\xi+(\lambda-\xi)t+(\mu-\lambda)s)\,ds\,dt
\end{align}
(the first integral can be obtained from the second one by substituting $x=t-s$, $dx=-ds$),
so it is enough to prove
\begin{align}
\label{eq:tmhp}
&\int_0^1\int_0^t t^{m-1}h(\xi+(\lambda-\xi)t+(\mu-\lambda)s)\,ds\,dt\\
\nonumber
&=\frac1m\left(\frac{\xi-\lambda}{\mu-\lambda}\right)\int_0^1 (1-t^m)h(\xi+(\lambda-\xi)t)\,dt\\
\nonumber
&\quad+\frac1m\left(\frac{\mu-\xi}{\mu-\lambda}\right)\int_0^1 (1-t^m)h(\xi+(\mu-\xi)t)\,dt.
\end{align}

First we assume that all points $\lambda,\xi,\mu$ are distinct. It is simple to verify the equality
\[t^{m-1}h(\xi+(\lambda-\xi)t+(\mu-\lambda)s)=\frac{\partial Q}{\partial t}(s,t)-\frac{\partial P}{\partial s}(s,t),\] where
\begin{align*}
&Q(s,t)=\frac1m t^m h(\xi+(\lambda-\xi)t+(\mu-\lambda)s),\\
&P(s,t)=\frac1m\frac{\lambda-\xi}{\mu-\lambda} t^m h(\xi+(\lambda-\xi)t+(\mu-\lambda)s).
\end{align*}
By Green's theorem applied to the function
\[(s,t)\mapsto t^{m-1}h(\xi+(\lambda-\xi)t+(\mu-\lambda)s),\] we have
\[\iint_D \left(\frac{\partial Q}{\partial t}-\frac{\partial P}{\partial s}\right)\,ds\,dt
=\oint_{\partial D} Q\,ds+\oint_{\partial D} P\,dt,\]
where $D\subset\C$ is the triangular region with vertices at the points $\lambda,\xi,\mu$ with the positively oriented boundary $\partial D$ given by the equations $s=0$, $t=1$, and $s=t$ (the boundary is also simple, closed, and piecewise smooth).
Thus,
\begin{align}
\label{QP1}
\nonumber
&\oint_{\partial D} Q\,ds+\oint_{\partial D} P\,dt\\
\nonumber
&\;=\frac1m\left(\frac{\lambda-\xi}{\mu-\lambda}\right)\int_0^1 t^m h(\xi+(\lambda-\xi)t)\,dt
-\frac1m\int_0^1 t^m h(\xi+(\mu-\xi)t)\,dt\\
&\quad-\frac1m\left(\frac{\lambda-\xi}{\mu-\lambda}\right)\int_0^1 t^m h(\xi+(\mu-\xi)t)\,dt+\frac1m\int_0^1 h(\lambda+(\mu-\lambda)s)\,ds.
\end{align}
We apply Lemma \ref{thm:BaseDecomp} (with $m=0$) to rewrite the latter integral in the form
\begin{align}
\label{QP2}
&\frac1m\int_0^1 h(\lambda+(\mu-\lambda)s)\,ds\\
\nonumber
&\quad=\frac1m\left(\frac{\xi-\lambda}{\mu-\lambda}\right)\int_0^1 h(\lambda+(\xi-\lambda)t)\,dt+
\frac1m\left(\frac{\mu-\xi}{\mu-\lambda}\right)\int_0^1 h(\xi+(\mu-\xi)t)\,dt.
\end{align}
Substituting \eqref{QP2} into \eqref{QP1} and combining the second and third terms of \eqref{QP1} gives
\eqref{eq:tmhp}.

When $\lambda=\xi$, upon changing the order of integration, the left hand side of \eqref{eq:tmhp} equals
\begin{align*}
\int_0^1\int_s^1 t^{m-1}h(\xi+(\mu-\xi)s)\,dt\,ds
=\frac1m\int_0^1(1-s^m)h(\xi+(\mu-\xi)s)\,ds,
\end{align*}
which also equals the right hand side of \eqref{eq:tmhp}.

When $\xi=\mu$, the left hand side of \eqref{eq:tmhp} equals
\[\int_0^1\int_0^t t^{m-1}h(\xi+(\lambda-\xi)(t-s))\,ds\,dt.\]
We make substitution $x=t-s$ and change the order of integration to obtain $\int_0^1\int_x^1 t^{m-1}h(\xi+(\lambda-\xi)x)\,dt\,dx$, which coincides with the right hand side of \eqref{eq:tmhp}.
\end{proof}

\begin{lemma}
\label{thm:tkh}
Let $\kappa \in (0,1]$, $\lambda,\xi,\mu\in\T$, with $\lambda\neq\xi$, and 
let $h$ be a polynomial.
Then,
\begin{align}
\label{eq:skh}
&\int_0^\kappa\int_0^t s^{k-1}h(\kappa\xi+(\lambda-\xi)t+(\mu-\lambda)s)\,ds\,dt\\
\nonumber
&=\frac1k\left(\frac{\mu-\lambda}{\xi-\lambda}\right)\int_0^\kappa t^k h(\kappa\lambda+(\mu-\lambda)t)\,dt\\
\nonumber
&\quad-\frac1k\left(\frac{\mu-\xi}{\xi-\lambda}\right)\int_0^\kappa t^k h(\kappa\xi+(\mu-\xi)t)\,dt.
\end{align}
\end{lemma}

\begin{proof}
We prove \eqref{eq:skh} first for the case $\kappa=1$ and then make the change of variables as in Lemma \ref{thm:tmh}.

First, we assume that all points $\lambda,\xi,\mu$ are distinct.
Note that
\[s^{k-1}h(\xi+(\lambda-\xi)t+(\mu-\lambda)s)=\frac{\partial Q}{\partial t}(s,t)-\frac{\partial P}{\partial s}(s,t),\] where
\begin{align*}
&Q(s,t)=-\frac{1}{k}\left(\frac{\mu-\lambda}{\lambda-\xi}\right)s^k h(\xi+(\lambda-\xi)t+(\mu-\lambda)s),\\
&P(s,t)=-\frac1k s^k h(\xi+(\lambda-\xi)t+(\mu-\lambda)s).
\end{align*}
By Green's theorem we obtain
\begin{align*}
&\int_0^1\int_0^t s^{k-1}h(\xi+(\lambda-\xi)t+(\mu-\lambda)s)\,ds\,dt\\
&\quad=\frac1k\left(\frac{\mu-\lambda}{\lambda-\xi}\right)\int_0^1 t^k h(\xi+(\mu-\xi)t)\,dt
+\frac1k\int_0^1 t^k h(\xi+(\mu-\xi)t)\,dt\\
&\quad\quad-\frac1k\left(\frac{\mu-\lambda}{\lambda-\xi}\right)\int_0^1 s^k h(\lambda+(\mu-\lambda)s)\,ds,
\end{align*}
which equals \eqref{eq:skh}.

The case $\mu=\lambda$ follows upon evaluating the inner integral on the left hand side of \eqref{eq:skh}. In the case $\mu=\xi$, the formula can be obtained by substitution and change of the order of integration similarly to how it was done in Lemma \ref{thm:tmh}.
\end{proof}


\begin{lemma}
\label{thm:tmkh}
If $n\geq 2$ and $m,k\in\N$, then

(i) for $\lambda_0\neq\lambda_1$,
\begin{align*}
&\phi_{n,h,m-1,0}(\lambda_0,\lambda_1,\lambda_2,\dots,\lambda_n)\\
&\;=\frac{1}{m}\left(1+\frac{\lambda_2-\lambda_1}{\lambda_1-\lambda_0}\right)
\big(\phi_{n-1,h,m,0}(\lambda_0,\lambda_2,\dots,\lambda_n)
-\phi_{n-1,h,0,m}(\lambda_0,\lambda_2,\dots,\lambda_n)\big)\\
&\quad +\frac{1}{m}\frac{\lambda_2-\lambda_1}{\lambda_1-\lambda_0}
\big(\phi_{n-1,h,0,m}(\lambda_1,\lambda_2,\dots,\lambda_n)
-\phi_{n-1,h,m,0}(\lambda_1,\lambda_2,\dots,\lambda_n)\big),
\end{align*}
(ii) for $\lambda_1\neq\lambda_2$,
\begin{align*}
&\phi_{n,h,0,k-1}(\lambda_0,\lambda_1,\lambda_2,\dots,\lambda_n)\\
&\;=
\frac{1}{k}\left(1+\frac{\lambda_1-\lambda_0}{\lambda_2-\lambda_1}\right)
\phi_{n-1,h,0,k}(\lambda_0,\lambda_2,\lambda_3,\dots,\lambda_n)\\
&\quad-\frac{1}{k}\frac{\lambda_1-\lambda_0}{\lambda_2-\lambda_1}
\phi_{n-1,h,0,k}(\lambda_0,\lambda_1,\lambda_3,\dots,\lambda_n).
\end{align*}
(iii) for $\lambda_0=\lambda_1=\dots=\lambda_n$,
\[\phi_{n,h,m-1,k-1}(\lambda_0,\dots,\lambda_0)=c_{n,m,k}h(\lambda_0).
\]
\end{lemma}

\begin{proof}
The formula in (iii) is obtained by evaluating the 
integrals in \eqref{phinhmk}.
For $n=2$, (i) and (ii) are the assertions of Lemmas \ref{thm:tmh} and \ref{thm:tkh}, respectively. The case $n>2$ is proved by reduction to the case $n=2$. We prove (i); (ii) can be proved almost verbatim.

By making the substitution
\[s=s_0,\, t=s_0+s_1,\, t_3=\sum_{j=0}^2 s_j,\, \dots,\, t_n=\sum_{j=0}^{n-1} s_j,\, 1=\sum_{j=0}^n s_j\] in the integral in \eqref{phinhmk}, we obtain
\begin{align*}
&\phi_{n,h,m-1,0}(\lambda_1,\lambda_0,\lambda_2,\dots,\lambda_n)\\
&\quad=\int_{S_n^1}(s_0+s_1)^{m-1}
h(s_0\lambda_0+s_1\lambda_1+s_2\lambda_2+s_3\lambda_3+\dots+s_n\lambda_n)\,d\sigma_n,
\end{align*}
where the simplex
\[S_n^\kappa=\bigg\{(s_0,\dots,s_n)\in\R^{n+1}: \sum_{j=0}^n s_j=\kappa,\, s_j\geq 0,\, 0\leq j\leq n\bigg\}\] is equipped with the Lebesgue surface measure $d\sigma_n$ defined by
\[\int_{S_n^\kappa}\phi(s_0,\dots,s_n)\,d\sigma_n=
\int_{R_n^\kappa}\phi\bigg(s_0,\dots,s_{n-1},\kappa-\sum_{j=0}^{n-1} s_j\bigg)\,dv_n\]
for every continuous function $\phi:\R^{n+1}\mapsto\C$, where $dv_n$ is the Lebesgue measure on $\R^n$ and
\[R_n^\kappa=\bigg\{(s_0,\dots,s_n)\in\R^{n+1}: \sum_{j=0}^n s_j\leq\kappa,\, s_j\geq 0,\, 0\leq j\leq n\bigg\}.\] If we set $\kappa=1-\sum_{j=3}^n s_j$, then we can split the integral over $S_n^1$ into the repeated integral $\int_{R_{n-2}^1}\,ds_3\dots ds_n\int_{S_2^\kappa}\,d\sigma_n$. Therefore, if we set $s=s_1$ and $t=s_0+s_1$, we obtain
\begin{align*}
&\phi_{n,h,m-1,0}(\lambda_0,\lambda_1,\lambda_2,\dots,\lambda_n)\\
&\quad=\int_{R_{n-2}^1}\,ds_3\dots ds_n\int_0^\kappa\int_0^t t^{m-1}
h(\kappa\lambda_2+(\lambda_1-\lambda_2)t+(\lambda_0-\lambda_1)s)\,ds\, dt.
\end{align*}
By Lemma \ref{thm:tmh}, the latter equals
\begin{align*}
&\frac1m\left(\frac{\lambda_2-\lambda_0}{\lambda_1-\lambda_0}\right)
\int_{R_{n-2}^1}\,ds_3\dots ds_n\int_0^\kappa (\kappa^m-t^m)
h(\kappa\lambda_2+(\lambda_0-\lambda_2)t)\, dt\\
&\quad+\frac1m\left(\frac{\lambda_1-\lambda_2}{\lambda_1-\lambda_0}\right)
\int_{R_{n-2}^1}\,ds_3\dots ds_n\int_0^\kappa (\kappa^m-t^m)
h(\kappa\lambda_2+(\lambda_1-\lambda_2)t)\, dt.
\end{align*}
Making the substitution $s_0=t$, $s_2=\kappa-t$ in the first and $s_1=t$, $s_2=\kappa-t$ in the second integral, respectively, we obtain
\begin{align*}
&\phi_{n,h,m-1,0}(\lambda_0,\lambda_1,\lambda_2,\dots,\lambda_n)\\
&=\frac1m\left(\frac{\lambda_2-\lambda_0}{\lambda_1-\lambda_0}\right)
\int_{R_{n-2}^1}\,ds_3\dots ds_n\int_{S_1^\kappa} ((s_0+s_2)^m-s_0^m)
h(s_0\lambda_0+s_2\lambda_2+\dots+s_n\lambda_n)\, d\sigma_1\\
&+\frac1m\left(\frac{\lambda_1-\lambda_2}{\lambda_1-\lambda_0}\right)
\int_{R_{n-2}^1}\,ds_3\dots ds_n\int_{S_1^\kappa} ((s_1+s_2)^m-s_1^m)
h(s_1\lambda_1+s_2\lambda_2+\dots+s_n\lambda_n)\, d\sigma_1.
\end{align*}
By substituting $t=s_0$, $t_3=s_0+s_2$, \dots, $t_n=s_0+\sum_{j=2}^{n-1} s_j$, $1=s_0+\sum_{j=2}^{n} s_j$
in the first and
$t=s_1$, $t_3=s_1+s_2$, \dots, $t_n=\sum_{j=1}^{n-1} s_j$, $1=\sum_{j=1}^{n} s_j$ in the second integral, respectively, we obtain
\begin{align*}
&\phi_{n,h,m-1,0}(\lambda_0,\lambda_1,\lambda_2,\dots,\lambda_n)\\
&=\frac1m\left(\frac{\lambda_2-\lambda_0}{\lambda_1-\lambda_0}\right)
\int_0^1\int_0^{t_n}\dots\int_0^{t_3}(t_3^m-t^m)\\
&\quad\quad\quad\quad\quad \cdot h(\lambda_n+(\lambda_{n-1}-\lambda_n)t_n
+\dots+(\lambda_2-\lambda_3)t_3+(\lambda_0-\lambda_2)t)\,dt\,dt_3\dots dt_n\\
&+\frac1m\left(\frac{\lambda_1-\lambda_2}{\lambda_1-\lambda_0}\right)
\int_0^1\int_0^{t_n}\dots\int_0^{t_3}(t_3^m-t^m)\\
&\quad\quad\quad\quad\quad \cdot h(\lambda_n+(\lambda_{n-1}-\lambda_n)t_n
+\dots+(\lambda_2-\lambda_3)t_3+(\lambda_1-\lambda_2)t)\,dt\,dt_3\dots dt_n,
\end{align*}
which completes the proof of (i).
\end{proof}

\begin{theorem}
\label{thm:IndStep}
Let $n,m\in\N$ and let $1<\alpha,\alpha_j<\infty$, for $j=1,\dots,n$, be such that $0<\frac{1}{\alpha}=\frac{1}{\alpha_1}+\cdots +\frac{1}{\alpha_n}<1$. For $\phi_{n,h,p,q}$ as in \eqref{phinhmk} and $(p,q)\in\{(m-1,0),(0,m-1)\}$, the transformation $T_{\phi_{n,h,p,q}}: S^{\alpha_1}\times\dots\times S^{\alpha_n}\mapsto S^\alpha$ is bounded and
\[\big\|T_{\phi_{n,h,p,q}}\big\|\leq c_{n,m,\alpha_1,\dots,\alpha_n}\|h\|_\infty.\]
\end{theorem}

First we prove the boundedness of $T_{\phi_{n,h,p,q}}$ on the diagonal set.

\begin{lemma}
\label{diagonal} Assume the notation of Theorem \ref{thm:IndStep} and let
\[A_0^{(n)}:=\big\{(\lambda_0,\dots,\lambda_n):\, \lambda_0=\dots=\lambda_n\in \{z_j\}_{j=0}^{N-1}\big\}.\]
The transformation $T^{A_0^{(n)}}_{\phi_{n, h, m-1,0}}$ is bounded on $S^{\alpha_1}\times\dots\times S^{\alpha_n}$, with the bound as in Theorem \ref{thm:IndStep}.
\end{lemma}

\begin{proof}
We prove a more general result. Let $\phi$ be a bounded Borel function on $\T^{n+1}$. Then, the polylinear operator $$ \Delta_\phi(x_1, \ldots, x_n) := \sum_{j=0}^{N-1}\phi(z_j,z_j,\dots,z_j) E_j
  x_1 E_j x_2 \cdot \ldots \cdot x_n E_j $$ is bounded on $S^{\alpha_1}\times\dots\times S^{\alpha_n}$.

Since \[\Delta_\phi(x_1, \ldots, x_n)=\sum_{j=0}^{N-1}\phi(z_j,z_j,\dots,z_j)E_j\sum_{k=0}^{N-1} E_k
  x_1 E_k x_2 \cdot \ldots \cdot x_n E_k,\] it is enough to prove the boundedness of $\Delta_\phi$ with  $\phi\equiv 1$.

Consider the unitary $$ U_t = \sum_{j=0}^{N-1} e^{2\pi i j t} E_j,\ \ 0 \leq t \leq 1. $$ We now observe
  that $$ \Delta_1(x_1,
  \ldots, x_n) = \int_0^1 \ldots \int_0^1 \left[ \prod_{j = 1}^n
    U^*_{t_j} x_j U_{t_j} \right]\, dt_1 \ldots dt_n $$
due to orthogonality of the trigonometric functions.
Thus, the estimate $$ \left\| \Delta_1(x_1, \ldots x_n )\right\|_\alpha \leq \left\|
    x_1 \right\|_{\alpha_1} \cdot \ldots \cdot \left\| x_n
  \right\|_{\alpha_n} $$ follows.
\end{proof}

\begin{proof}[Proof of Theorem \ref{thm:IndStep}]
The case $n=1$ is proved in Theorem \ref{thm:IndBase}.
We prove the bound for $T_{\phi_{n,h,m-1,0}}$, $n>1$, by induction on $n$; the case of $T_{\phi_{n,h,0,m-1}}$ is completely analogous.

Assume that the transformation $T_{\phi_{n-1, h, k, l}}$, where $(k,l)\in\{(m,0),(0,m)\}$, is bounded on $S^{\alpha_1}\times\dots\times S^{\alpha_n}$ with norm no greater than $c_{n,m,\alpha_1,\dots,\alpha_n}\|h\|_\infty$.
Note that it is enough to show that
\begin{equation*}
    \sup_{x_j \in S^{\alpha_j},\; \|x_j\|_{\alpha_j}\leq 1,\; 0\leq j\leq n}\left| \tr
      \left( x_0 T_{\phi_{n, h, m-1, 0}} \left(x_1,\dots,x_n \right) \right)
    \right| \leq c_{n,m,\alpha_1,\dots,\alpha_n}\|h\|_\infty,
  \end{equation*}
  where $1<\alpha_0<\infty$ is such that $1=\frac{1}{\alpha_0}+\frac{1}{\alpha_1}+\cdots+\frac{1}{\alpha_n}$.

We shall use the boundedness of $T_{\phi_{n-1, h, k, l}}$ and the decomposition of Lemma \ref{thm:tmkh} to prove the boundedness of $T_{\phi_{n, h, m-1, 0}}$. Denote
\[Q^{(n)}_k:=\left\{z\in\T: \arg(z)\in \left[\frac{2\pi k}{n+2},\frac{2\pi (k+1)}{n+2}\right)\right\},
\quad k=0,\dots,n+1,\] and let $N$ from Definition \ref{MOI-def} be divisible by $n+2$. (Here $\arg(z)$ denotes the principal value of the argument of the complex number $z$.)
By additivity of the multiple operator integral, we have
\begin{align}
\label{Tsum}
T_{\phi_{n, h, m-1, 0}}=\sum_{k_0,k_1,\dots,k_n\in \{0,1,\dots,n+1\}}
T_{\phi_{n, h, m-1, 0}}^{Q^{(n)}_{k_0}\times Q^{(n)}_{k_1}\times\dots\times Q^{(n)}_{k_n}},
\end{align}
where the number of summands is, clearly, determined by $n$. We shall estimate separately each of the terms in
\eqref{Tsum}.

The boundedness of $T^{A_0^{(n)}}_{\phi_{n, h, m-1,0}}$ on $S^{\alpha_1}\times\dots\times S^{\alpha_n}$ follows from Lemma \ref{diagonal},
so we shall consider $T_{\phi_{n, h, m-1, 0}}$ only away of the diagonal set $A_0^{(n)}$.

We shall estimate the norm of
\begin{align}
\label{Ttheta}
T_{\phi_{n, h, m-1, 0}}^{Q^{(n)}_{k_0}\times Q^{(n)}_{k_1}\times\dots\times Q^{(n)}_{k_n}},
\end{align} by applying the following method.
\medskip

Similarly to Lemma
\ref{thm:FactorDecomp}, we have
\[\frac{\lambda_2-\lambda_1}{\lambda_1-\lambda_0}=\left(\frac{|\lambda_1-\lambda_0|}{\lambda_1-\lambda_0}\right)
\left(\frac{\lambda_2-\lambda_1}{|\lambda_2-\lambda_1|}\right)\int_R g_\delta(s)\left(\frac{|\lambda_2-\lambda_1|^{\i s}}{|\lambda_1-\lambda_0|^{\i s}}\right)\,ds,\]
whenever $\frac{|\lambda_2-\lambda_1|}{|\lambda_1-\lambda_0|}\leq \delta=\delta_n$. Let
\begin{align*}
&A:=\big\{(\lambda_0,\dots,\lambda_n):\, \lambda_0,\dots,\lambda_n\in \{z_j\}_{j=0}^{N-1},\; \lambda_0\neq\lambda_1\big\},\\
&A^+:=\left\{(\lambda_0,\dots,\lambda_n)\in A:\,
\frac{|\lambda_2-\lambda_1|}{|\lambda_1-\lambda_0|}\leq \delta\right\}.
\end{align*} Let $D$ be a Borel subset of $A^+$. Denote by $D_{j_0,\dots,j_k}$ the projection of $D$ onto the coordinates $j_0,\dots,j_k$.
Denote
\begin{align*}
\psi_{0,s}(\lambda_0,\lambda_1):=
\left(\frac{|\lambda_1-\lambda_0|}{\lambda_1-\lambda_0}\right)|\lambda_1-\lambda_0|^{-\i s},\quad
\psi_{2,s}(\lambda_1,\lambda_2):=
\left(\frac{\lambda_2-\lambda_1}{|\lambda_2-\lambda_1|}\right)|\lambda_2-\lambda_1|^{\i s}.
\end{align*}
By Lemma \ref{thm:tmkh} and Lemma \ref{MOI-algebra} \eqref{MOI-A-product} and \eqref{MOI-A-composition},
\begin{align}
\label{reasoning}
&m T_{\phi_{n, h, m-1, 0}}^{D}(x_1,\dots,x_n)\\ \nonumber
&\;=T_{\phi_{n-1, h, m, 0}}^{D}(x_1,\dots,x_n)-T_{\phi_{n-1, h, 0, m}}^{D}(x_1,\dots,x_n) \displaybreak[2] \\ \nonumber
&\quad+\int_\R T_{\phi_{n-1, h, m, 0}}^{D}\big(T_{\psi_{0,s}\psi_{2,s}}^{D_{0,1,2}}(x_1,x_2),x_3,\dots,x_n\big) g_\delta(s)\,ds
\displaybreak[2] \\ \nonumber
&\quad-\int_\R T_{\phi_{n-1, h, 0, m}}^{D}\big(T_{\psi_{0,s}\psi_{2,s}}^{D_{0,1,2}}(x_1,x_2),x_3,\dots,x_n\big) g_\delta(s)\,ds \displaybreak[2] \\ \nonumber
&\quad+\int_\R T_{\psi_{0,s}}^{D_{0,1}}(x_1) \cdot T_{\phi_{n-1, h, m, 0}}^{D}\big(T_{\psi_{2,s}}^{D_{1,2}}(x_2),x_3,\dots,x_n\big) g_\delta(s)\,ds\\
\nonumber
&\quad-\int_\R T_{\psi_{0,s}}^{D_{0,1}}(x_1) \cdot T_{\phi_{n-1, h, 0, m}}^{D}\big(T_{\psi_{2,s}}^{D_{1,2}}(x_2),x_3,\dots,x_n\big) g_\delta(s)\,ds.
\end{align}
Thus, in order to claim the boundedness of $T_{\phi_{n, h, m-1, 0}}^{D}$, it is enough to prove the boundedness of $T_{\phi_{n-1, h, k, l}}^{D}$, with $(k,l)\in \{(m,0),(0,m)\}$, and the boundedness of $T_{\psi_{0,s}}^{D_{0,1}}$ and $T_{\psi_{2,s}}^{D_{1,2}}$.
\medskip

Now we apply the general method to estimate the norm of \eqref{Ttheta}. There are two principal cases. One is when there exists an index $i\in\{0,1,\dots,n\}$ such that\footnote{Here the increment and decrement of the index $i$ is understood modulo $n$, that is, if $i=n$, then $i+1=0$ and if $i=0$, then $i-1=n$.} $|k_{i+1}-k_i|\geq 2$ (and, hence, $|z-w|\geq c_n>0$ for $z\in Q_{k_{i+1}}^{(n)}$, $w\in Q_{k_i}^{(n)}$) and the other is when $|k_{i+1}-k_i|\leq 1$ for all $i$. In the latter case,\footnote{If $n$ is even, then a typical example is $\begin{cases}k_i=i,& i\leq\frac{n}{2}\\ k_i=n-i+1,& i>\frac{n}{2}.\end{cases}$} there is $a\in(0,\pi]$ such that $\arg(z)\subseteq [a,a+\pi]$ whenever $z\in Q_{k_i}^{(n)}$, for each $i$.
Thus, in this case we have the inequality $|z_{j_1}-z_{j_0}|> |z_{j_2}-z_{j_1}|$ whenever $z_{j_0},z_{j_1},z_{j_2}\in Q^{(n)}_{k_0}\cup Q^{(n)}_{k_1}\cup\dots\cup Q^{(n)}_{k_n}$ and $j_0\leq j_2<j_1$.
\medskip

{\it Case 1:} there exists $i$ such that $|k_{i+1}-k_i|\geq 2$.

As noted above, shifting the variables does not affect the norm of \eqref{Ttheta}, so it is enough to consider the subcase $i=0$.

We apply the reasoning \eqref{reasoning} with
\[D=Q^{(n)}_{k_0}\times Q^{(n)}_{k_1}\times\dots\times Q^{(n)}_{k_n},\quad D_{0,1}=Q^{(n)}_{k_0}\times Q^{(n)}_{k_1},\quad\text{ and}\quad
D_{1,2}=Q^{(n)}_{k_1}\times Q^{(n)}_{k_2}.\] Set $\mathcal{Q}_{k_0}^{(n)}:=E\big(Q_{k_0}^{(n)}\big)$. Then,
the operator
\[T_{\phi_{n-1, h, k, l}}^{D}(x_1,x_2,\dots,x_n)
=T_{\phi_{n-1, h, k, l}}\big(\mathcal{Q}_{k_0}^{(n)}x_1 \mathcal{Q}_{k_1}^{(n)},\mathcal{Q}_{k_1}^{(n)} x_2\mathcal{Q}_{k_2}^{(n)},\dots,\mathcal{Q}_{k_{n-1}}^{(n)}x_n\mathcal{Q}_{k_n}^{(n)}\big)\] is bounded by the induction notations and boundedness of the projections $\mathcal{Q}_{k_i}^{(n)}$, $i=0,\dots,n$; the operators
$T_{\psi_{0,s}}^{D_{0,1}}$ and $T_{\psi_{2,s}}^{D_{1,2}}$ are bounded by Lemma \ref{thm:KPSS} and Lemma
\ref{MOI-algebra} \eqref{MOI-A-composition}; the operator
\[T_{\psi_{0,s}\psi_{2,s}}^{D_{0,1,2}}(x_1,x_2)=
T_{\psi_{0,s}}^{D_{0,1}}(x_1)T_{\psi_{2,s}}^{D_{1,2}}(x_2)\] is bounded by Lemma \ref{MOI-algebra} \eqref{MOI-A-product}. This completes the proof of Case 1.
\medskip

We split the case ``$|k_{i+1}-k_i|\leq 1$ for all $i$" into two subcases below.

{\it Case 2:} $k_0=k_1=\dots=k_n$.

We adjust the argument of Theorem \cite[Theorem 5.3]{PSS} and demonstrate only the case $k_0=0$. Let $\epsilon=(\epsilon_1,\epsilon_2,\dots,\epsilon_n)\in\{-1,1\}^n$ and define
\begin{align*}
K_\epsilon:=\big\{(z_{j_0},\dots,z_{j_n})&\in Q_0^{(n)}\times\dots\times Q_0^{(n)}:\; \\
& j_{i-1}\leq j_i\; \text{ if }\; \epsilon_i=1;\; j_{i-1}> j_i\; \text{ if }\; \epsilon_i=-1,\;
1 \leq i\leq n\big\}.
\end{align*}
The space $Z:=\big\{(z_{j_0},\dots,z_{j_n})\in Q_0^{(n)}\times\dots\times Q_0^{(n)}\big\}\setminus A_0$
splits into the disjoint union of $2^n$ sets $K_\epsilon$, where $\epsilon\in\{-1,1\}^n$. There is an index $i_\epsilon$ such that $j_{i_\epsilon-1}\leq j_{i_\epsilon}$ and $j_{i_\epsilon}> j_{i_\epsilon+1}$. By fixing $j_\epsilon$, we further split $K_\epsilon$ into subsets $K_{\epsilon,d}$, $d=0,1$, where
\begin{align*}
&K_{\epsilon,0}:=\big\{(z_{j_0},\dots,z_{j_n})\in K_\epsilon:\; j_{i_\epsilon-1}\leq j_{i_\epsilon+1}\big\},\\
&K_{\epsilon,1}:=\big\{(z_{j_0},\dots,z_{j_n})\in K_\epsilon:\; j_{i_\epsilon-1}> j_{i_\epsilon+1}\big\}.
\end{align*}
The space $Z$ splits into the disjoint union of $2^{n+1}$ sets $K_{\epsilon,i}$ and, hence,
\[T_{\phi_{n,h,m-1,0}}^Z=\sum_\epsilon\sum_{d=0,1}T_{\phi_{n,h,m-1,0}}^{K_{\epsilon,d}}.\] For fixed $\epsilon$ and $i_\epsilon$,
\[(z_{j_0},\dots,z_{j_n})\in K_{\epsilon,d}\Rightarrow
\begin{cases}
j_{i_\epsilon-1}\leq j_{i_\epsilon+1}<j_{i_\epsilon},\; &\text{ if }\; d=0,\\
j_{i_\epsilon+1}< j_{i_\epsilon-1}\leq j_{i_\epsilon},\; &\text{ if }\; d=1.
\end{cases}\]
By shifting and also reversing if $i=1$ the enumeration of the variables (as in Lemma \ref{MOI-algebra}  \eqref{MOI-A-involution} and \eqref{MOI-A-duality}), we may assume that $j_0\leq j_2<j_1$. We apply the reasoning \eqref{reasoning} with
\[D=\big\{(z_{j_0},\dots,z_{j_n})\in Q_0^{(n)}\times\dots\times Q_0^{(n)}:\;
j_0\leq j_2<j_1\big\},\] which equals $K_{\epsilon,0}$ with $i_\epsilon=1$.
Let $x_1^{UT}$ denote the strictly upper triangular truncation of $x_1$ and $x_2^{LT}$ the strictly lower triangular truncation of $x_2$ with respect to the family $\{E_j\}_{j=0}^{N/(n+2)-1}$ (defined in Definition \ref{MOI-def}).
Then,
\[T_{\phi_{n-1, h, k, l}}^{D}(x_1,x_2,\dots,x_n)
=T_{\phi_{n-1, h, k, l}}\big(\mathcal{Q}_0^{(n)}x_1^{UT} \mathcal{Q}_0^{(n)},\mathcal{Q}_0^{(n)} x_2^{LT}\mathcal{Q}_0^{(n)},\dots,\mathcal{Q}_0^{(n)}x_n\mathcal{Q}_0^{(n)}\big),\]
\[T_{\psi_{0,s}}^{D_{0,1}}(x_1)=T_{\psi_{0,s}}(\mathcal{Q}_0^{(n)}x_1^{UT}\mathcal{Q}_0^{(n)}),\quad T_{\psi_{2,s}}^{D_{1,2}}(x_2)=T_{\psi_{2,s}}(\mathcal{Q}_0^{(n)}x_1^{LT}\mathcal{Q}_0^{(n)}),\] and
\[T_{\psi_{0,s}\psi_{2,s}}^{D_{0,1,2}}(x_1,x_2)=
T_{\psi_{0,s}}^{D_{0,1}}(x_1)T_{\psi_{2,s}}^{D_{1,2}}(x_2)\]
are bounded by Lemmas \ref{thm:KPSS} and \ref{MOI-algebra} and the boundedness of the triangular truncation.
\medskip

{\it Case 3:} $|k_{i+1}-k_i|\leq 1$ for all $i$ and $|k_{j+1}-k_j|=1$ for some $j$.

In this case, there exists an index $j$ such that
the sets $Q_{k_{j-1}}^{(n)}=Q_{k_{j+1}}^{(n)}$ and $Q_{k_j}^{(n)}$ are disjoint. By Lemma \ref{MOI-algebra} \eqref{MOI-A-duality}, it is enough to consider the subcase $j=1$.
Let $(z_{j_0},z_{j_1},z_{j_2})\in Q_{k_0}^{(n)}\times Q_{k_1}^{(n)}\times Q_{k_2}^{(n)}$. If $k_0=k_2<k_1$, then we have $j_0\leq j_2<j_1$ and, hence, $\frac{|z_{j_2}-z_{j_1}|}{|z_{j_1}-z_{j_0}|}< 1$ (or $j_2\leq j_0<j_1$ and, hence,
$\frac{|z_{j_1}-z_{j_0}|}{|z_{j_2}-z_{j_1}|}< 1$); if $k_0=k_2>k_1$, then we have $j_1< j_2\leq j_0$
(or $j_1< j_0\leq j_2$).
Since the point $(\lambda_2,\lambda_1,\lambda_0,\lambda_n,\dots,\lambda_3)$ can be obtained from $(\lambda_0,\lambda_1,\lambda_2,\lambda_3,\dots,\lambda_n)$ by shifting and reversing
\begin{align*}
(\lambda_0,\lambda_1,\lambda_2,\lambda_3\dots,\lambda_n)
&\mapsto (\lambda_1,\lambda_2,\lambda_3,\dots,\lambda_n,\lambda_0)\mapsto
(\lambda_0,\lambda_n,\dots,\lambda_3,\lambda_2,\lambda_1)\\
&\mapsto (\lambda_2,\lambda_1,\lambda_0,\lambda_n,\dots,\lambda_3),
\end{align*} in view of Lemma \ref{MOI-algebra} \eqref{MOI-A-involution} and \eqref{MOI-A-duality}, it is enough to consider only the subset of $Q_{k_0}^{(n)}\times Q_{k_1}^{(n)}\times Q_{k_2}^{(n)}$ for which  $\frac{|z_{j_2}-z_{j_1}|}{|z_{j_1}-z_{j_0}|}< 1$. For this subset, we apply the reasoning \eqref{reasoning} similarly to how it was done in Case 2, where $x_1,x_2$ are replaced with their upper or lower triangular truncations (depending on whether we have $j_0\leq j_2<j_1$ or $j_1<j_2\leq j_0$).
\end{proof}

\begin{proof}[Proof of Theorem \ref{thm:MainEst}]
(i) On the strength of Lemmas \ref{tou} and \ref{tofinite}, we assume that $U_0$ is unitary with spectrum contained in $\{z_j\}_{j=0}^{N-1}$ and that $t_0=0$.
By Lemma \ref{der=dd},
\[\frac{d^{n}}{dt^{n}}\big(f(U_0+tV)\big)\big|_{t=0}=n!T_{f^{[n]}}(\underbrace{V, \ldots,
        V}_{n\text{ \rm times}}).\]
Hence, by Theorem \ref{thm:IndStep} applied to $T_{\phi_{n,f^{(n)},0,0}}$ and $\alpha_j=n$, $j=1,\dots,n-1$, we have \eqref{eq:MainEst}.

(ii) Applying Lemma \ref{thm:TrDer}, H\"{o}lder's inequality, and (i) 
completes the proof.
\end{proof}

\section{Proof of Theorem \ref{existence}}
\label{sec3}

We need the following formula computing the norm on the factor space $L^1(\T)/H^1(\T)$, where $H^1$ is the Hardy space $\{f\in L^1(\T): \hat f(n)=0,\text{ for all } n<0\}$.

\begin{lemma}(\cite[Lemma 5]{PS-circle})
\label{lemmaPS}
For every $f\in L^1(\T)$, the equality
\[\|f\|_{L^1/H^1}=\sup_{\|g\|_{H^\infty}\leq 1,\; g\in\mathcal{P}}\left|\int_\T g(z)f(z)\,dz\right|\] holds, where $\mathcal{P}$ is the set of all complex polynomials.
\end{lemma}

The next theorem extends the result of \cite[Theorem 6]{PS-circle} for $V\in S^2$ to perturbations in  $S^n$, $n\geq 3$. The proof of this general result relies on the estimate \eqref{eq:TrEst} of Theorem \ref{thm:MainEst} and on Lemma \ref{thm:TrDer}.

\begin{theorem}
\label{aux}
Assume Notations \ref{hyp}. Let $W\in S^n$. There is a function $\eta_n\in L^1(\T)$ depending on $n,U_0,V,W$ such that
\begin{align}
\label{avr-func}
\frac{1}{(n-1)!}\int_0^1(1-t)^{n-1}\,\tr\left(\frac{d^{n-1}}{dt^{n-1}}f'(U_0+tV)W\right)\,dt
=\int_\T f^{(n)}(z)\eta_n(z)\,dz,
\end{align} for every polynomial $f$. The class of all such functions $\eta_n$ corresponds to a unique element $[\eta_n]\in L^1/H^1$ satisfying
\begin{align}
\label{etabound}
\|[\eta_n]\|_{L^1/H^1}\leq c_n\|V\|_n^{n-1}\|W\|_n.
\end{align}
\end{theorem}

\begin{proof} The proof is split into two steps.

{\it Step 1.} We show that there is a measure $\nu_{n,W}$ with $\|\nu_{n,W}\|\leq c_n\|V\|_n^{n-1}\|W\|_n$ such that
\begin{align*}
\frac{1}{(n-1)!}\int_0^1(1-t)^{n-1}\,\tr\left(\frac{d^{n-1}}{dt^{n-1}}f'(U_0+tV)W\right)\,dt
=\int_\T f^{(n)}(z)\,d\nu_{n,W}(z).
\end{align*}
Let $A(\T)$ denote the space $A(\D)\cap C(\overline\D)$ equipped with the norm induced from the space $C(\overline\D)$. Consider the quotient space
\[\mX_n=A(\T)/\{f\in A(\T): f^{(n)}=0\},\] where the completion is taken with respect to the norm $\|f^{(n)}\|_\infty$ coming from the seminorm $f\mapsto \|f^{(n)}\|_\infty$ on $A(\T)$. The space $(\mX_n,[f]\mapsto \|f^{(n)}\|_\infty)$ is a Banach space and it is isometrically isomorphic to $A(\T)$ via the $n$th power of the differentiation operator. Consider the linear functional
\begin{align}
\label{phiWdef}
\phi_W(f^{(n)}):=\frac{1}{(n-1)!}\int_0^1(1-t)^{n-1}\,\tr\left(\frac{d^{n-1}}{dt^{n-1}}f'(U_0+tV)W\right)\,dt
\end{align}
on $\mX_n$, which is well defined because the right hand side of \eqref{phiWdef} equals zero if $f^{(n)}=0$. This follows from Lemma \ref{thm:PolDer} since $f^{(n)}=0$ implies that the degree of the polynomial $f$ is less than $n$. From Theorem \ref{thm:MainEst} (i) and H\"{o}lder's inequality, we obtain
\begin{align}
\label{phi(f)}
|\phi_W(f^{(n)})|\leq c_n \|f^{(n)}\|_\infty\|V\|_n^{n-1}\|W\|_n,
\end{align} that is, $\phi_W$ is continuous on
$\mX_n\simeq A(\T)$, which can be considered as a closed subspace of $C(\T)$. Thus, by the Riesz-Markov and Hahn-Banach theorems, there is a finite complex-valued measure $\nu_{n,W}$ on $\T$ such that
\begin{align}
\nonumber
&\phi_W(f^{(n)})=\int_\T f^{(n)}(z)\,d\nu_{n,W}(z),\\
\label{phinorm}
&\|\phi_W\|_{\mX_n^*}=\|\nu_{n,W}\|\leq c_n\|V\|_n^{n-1}\|W\|_n.
\end{align}

{\it Step 2.}
We show that any measure $\nu_{n,W}$ satisfying \eqref{phinorm} has an absolutely continuous anti-analytic part, that is, there is $\eta_n\in L^1(\T)$ such that $\hat\nu_n(k)=\hat\eta_n(k)$, for $k\leq -1$.

Firstly, we assume that $W\in S^1$. Integration by parts gives
\begin{align*}
&\frac{1}{(n-1)!}\int_0^1(1-t)^{n-1}\,\tr\left(\frac{d^{n-1}}{dt^{n-1}}f'(U_0+tV)W\right)\,dt\\
&\quad=\frac{1}{(n-2)!}\int_0^1(1-t)^{n-2}\,\tr\left(\frac{d^{n-2}}{dt^{n-2}}f'(U_0+tV)W\right)\,dt\\
&\quad\quad-\frac{1}{(n-1)!}\,\tr\left(\frac{d^{n-2}}{dt^{n-2}}f'(U_0+tV)\big|_{t=0}W\right)\\
&\quad=:\phi_{W,1}(f^{(n-1)})-\phi_{W,2}(f^{(n-1)}).
\end{align*}
Consider the functionals $\phi_{W,1}$ and $\phi_{W,2}$ defined on $\mX_{n-1}\simeq A(\T)$. By repeating the reasoning of Step 1, we derive existence of finite complex-valued measures $\mu_{n,1}$ and $\mu_{n,2}$ such that
\begin{align*}
\phi_{W,1}(f^{(n-1)})=\int_\T f^{(n-1)}(z)\,d\mu_{n,1}(z)
\end{align*}
and
\begin{align*}
\phi_{W,2}(f^{(n-1)})=\int_\T f^{(n-1)}(z)\,d\mu_{n,2}(z).
\end{align*}
Observing that the function $e^{\i\theta}\mapsto M_{n,1}(e^{\i\theta})=\mu_{n,1}(S_\theta)$ is bounded and measurable on $S_\theta=\{z\in\T: 0\leq\arg z<\theta\}$ and using integration by parts, we obtain
\begin{align*}
\phi_{W,1}(f^{(n-1)})&=f^{(n-1)}(e^{\i\theta})M_{n,1}(e^{\i\theta})\bigg|_0^{2\pi}-
\int_0^{2\pi}f^{(n)}(e^{\i\theta})M_{n,1}(e^{\i\theta})\,d\theta\\
&=f^{(n-1)}(1)\mu_{n,1}(\T)+\i\int_\T f^{(n)}(z)M_{n,1}(z)\,\frac{dz}{z}.
\end{align*}
Similarly, integrating by parts gives
\begin{align*}
\phi_{W,2}(f^{(n-1)})&=f^{(n-1)}(1)\mu_{n,2}(\T)+\i\int_\T f^{(n)}(z)M_{n,2}(z)\,\frac{dz}{z},
&\quad M_{n,2}(e^{\i\theta})=\mu_{n,2}(S_\theta).
\end{align*}
Note that $\phi_W=\phi_{W,1}-\phi_{W,2}$ and apply this equality to $f(z)=\frac{1}{(n-1)!}z^{n-1}$ to derive
\[\mu_{n,1}(\T)-\mu_{n,2}(\T)=\phi_W(0)=0.\]
Therefore,
\[\phi_W(f^{(n)})=\int_\T f^{(n)}(z)\eta_n(z)\,dz,\] where
\[\eta_n(z)=\frac{i}{z}(M_{n,1}(z)-M_{n,2}(z))\in L^\infty(\T).\]
Employing the estimate \eqref{phi(f)} and Lemma \ref{lemmaPS} gives
\begin{align*}
\|[\eta_n]\|_{L^1/H^1}&=\sup_{\|f^{(n)}\|_\infty\leq 1,\; f\in\mathcal{P}}\left|\int_\T f^{(n)}(z)\eta_n(z)\,dz\right|=\sup_{\|f^{(n)}\|_\infty\leq 1,\; f\in\mathcal{P}}|\phi_W(f^{(n)})|\\
&\leq c_n\|V\|_n^{n-1}\|W\|_n.
\end{align*}
This completes the proof of Theorem \ref{aux} in case $W\in S^1$.

Now assume $W\in S^n$. Let $\{W_k\}_{k=1}^\infty\subseteq S^1$ be such that $\|W_k\|_n\leq\|W\|_n$ and
$\lim_{k\rightarrow\infty}\|W-W_k\|_n=0$. Let $\{\eta_{n,k}\}_{k=1}^\infty$ be the sequence of functions constructed above with respect to the triples $(U_0,V,W_k)$. We have
\begin{align*}
\|[\eta_{n,k}]-[\eta_{n,m}]\|_{L^1/H^1}&=\|\phi_{W_k}-\phi_{W_m}\|_{\mathcal{X}_n^*}
=\|\phi_{W_k-W_m}\|_{\mathcal{X}_n^*}\\
&\leq c_n\|V\|_n^{n-1}\|W_k-W_m\|_n
\end{align*}
and, hence, $\{\eta_{n,k}\}_{k=1}^\infty$ is Cauchy in $L^1(\T)/H^1(\T)$. Let $[\eta_n]$ denote the limit of this sequence in $L^1(\T)/H^1(\T)$, where $\eta_n\in L^1(\T)$. This $[\eta_n]$ satisfies the assertions of Theorem \ref{aux}.
\end{proof}

We conclude with the proof of the existence of the higher order spectral shift function on the unit circle.

\begin{proof}[Proof of Theorem \ref{existence}]
We invoke the integral representation for the remainder
\begin{equation}
    \label{SSFRemainderTemp}
    R_n (f, U_0, V) = \frac 1{(n-1)!}\, \int_0^1 (1 - t)^{n-1}
    \frac {d^n}{dt^{n}} f \left( U_t \right) \, dt,
\end{equation}
which follows from~\cite[Theorem~1.43 and~1.45]{Schwartz}.
Thus, by Lemma \ref{thm:TrDer}, \[f^{(n)}\mapsto \tr\big(R_n (f, U_0, V)\big)\] coincides with the functional \eqref{avr-func} in Theorem \ref{aux}, where $W=V$. Hence, Theorem \ref{aux} implies existence of $[\eta_n]\in L^1/H^1$ such that \[\|[\eta_n]\|_{L^1/H^1}\leq c_n\|V\|_n^n\]
and such that every representative in the class $[\eta_n]$ satisfies \eqref{hossf}. By the definition of $L^1/H^1$ norm, for every $\epsilon>0$, there is a function $\eta_n\in L^1(\T)$ such that \eqref{ssfest} holds.
\end{proof}

\bibliographystyle{plain}

\end{document}